\documentclass[hidelinks,onefignum,onetabnum]{siamart220329}
\pdfoutput=1

\usepackage{color}

\usepackage{cleveref}
\usepackage{mathtools} % for \coloneqq

\usepackage{amsmath}
\DeclareMathOperator*{\argmax}{arg\,max} % need \usepackage{amsmath} % limits underneath 

\usepackage{lipsum}
\usepackage{amsfonts}
\usepackage{graphicx}
\usepackage{epstopdf}
\usepackage{algorithmic}

\ifpdf
  \DeclareGraphicsExtensions{.eps,.pdf,.png,.jpg}
\else
  \DeclareGraphicsExtensions{.eps}
\fi

\newsiamremark{remark}{Remark}
\newsiamremark{hypothesis}{Hypothesis}
\crefname{hypothesis}{Hypothesis}{Hypotheses}
\newsiamthm{claim}{Claim}

\headers{Original energy dissipative and MBP-preserving rescaled ETDRK schemes}{C. Quan, X. Wang, P. Zheng, and Z. Zhou}

\title{Maximum bound principle and original energy dissipation of arbitrarily high-order rescaled ETD Runge--Kutta schemes for Allen--Cahn equations 
}

\author{Chaoyu Quan\thanks{School of Science and Engineering,  The Chinese University of Hong Kong,  Shenzhen,  518172,Guangdong, People's Republic of China
    (\email{quanchaoyu@cuhk.edu.cn}).}
        \and Xiaoming Wang\thanks{Department of Mathematics and Statistics, Missouri University of Science and
Technology, Rolla, MO 65409, United States of America
(\email{xiaomingwang@mst.edu}).}
  \and Pinzhong Zheng\thanks{Department of Mathematics, Southern University of Science and Technology, Shenzhen, 518055, Guangdong, People's Republic of China (\email{zhengpinzhong@outlook.com}).}
  \and Zhi Zhou\thanks{Department of Applied Mathematics, The Hong Kong Polytechnic University, Kowloon, Hong Kong SAR, People's Republic of China
    (\email{zhizhou@polyu.edu.hk}).}}

\usepackage{amsopn}

\ifpdf
\hypersetup{
  pdftitle={Maximum bound principle and original energy dissipation of arbitrary high-order rescaled exponential time differencing Runge-Kutta schemes for Allen--Cahn equations},
  pdfauthor={C. Quan, X. Wang, P. Zheng, and Z. Zhou}
}
\fi

\begin{document}

\maketitle

% REQUIRED
\begin{abstract}
The energy dissipation law and the maximum bound principle  are two critical physical properties of the Allen--Cahn equations. While many existing time-stepping methods are known to preserve the energy dissipation law, most apply to a modified form of energy. In this work, we demonstrate that, when the nonlinear term of the Allen--Cahn equation is Lipschitz continuous, a class of arbitrarily high-order exponential time differencing Runge--Kutta (ETDRK) schemes preserve the original energy dissipation property, under a mild step-size constraint. Additionally, we guarantee the Lipschitz condition on the nonlinear term by applying a rescaling post-processing technique, which ensures that the numerical solution unconditionally satisfies the maximum bound principle. Consequently, our proposed schemes maintain both the original energy dissipation law and the maximum bound principle and can achieve arbitrarily high-order accuracy. We also establish an optimal error estimate for the proposed schemes. Some numerical experiments are carried out to verify our theoretical results.
\end{abstract}

% REQUIRED
\begin{keywords}
  Exponential time differencing Runge--Kutta method; energy dissipation law, maximum bound principle, Allen--Cahn equation.
\end{keywords}

% REQUIRED
\begin{MSCcodes}
65M06, 65M12, 65M15
\end{MSCcodes}

\section{Introduction} \label{section:intro} 
Phase field equations play an important role in modeling a wide array of free-boundary problems across diverse fields such as materials science, physics, and biology \cite{cahn1958free,allen1979microscopic,hawkins2012numerical}. In this work, we consider a popular phase-field model, namely the Allen--Cahn equation:
\begin{equation}\label{eq:AC}
    \left\{
    \begin{aligned}
       & u_{t}      = \varepsilon^2 \Delta u+f(u), &  & x \in \Omega, ~  t \in(0,T], \\
       & u(x,0)  =u^0({x}),                     &  & x \in \bar{\Omega},
    \end{aligned}
    \right.
\end{equation}
equipped with homogeneous Neumann boundary condition.
Here, $\Omega$ denotes an open, connected, and bounded domain within $\mathbb{R}^d (d=1,2,3)$, and $\Delta$ representing the Laplacian operator over $d$ dimensions. The unknown function $u$ denotes the phase variable, and the parameter $\varepsilon >0$ represents the inter-facial width. The nonlinear term $f(u)=-F^{\prime}(u)$, where $F$ is a double-well potential with two wells at $\pm\beta$ for some $\beta>0$.
 
A notable feature of the Allen--Cahn equation is the \textit{maximum bound principle} (MBP), i.e., if the initial values are within $\beta$ in absolute value, the solution remains bounded by $\beta$ at all times. Furthermore, this model satisfies the so-called \textit{energy dissipation law}, because \cref{eq:AC} can be viewed as an $L^2$ gradient flow with respect to the energy functional:
\begin{equation}\label{eq:E}
    E(u)=\int_{\Omega}\left(\frac{\varepsilon^2}{2}|\nabla u|^2+F(u)\right) \mathrm{~d} x.
\end{equation}
The energy dissipation law is more precisely formulated as:
\begin{equation}\label{eq:dE/dt}
    \frac{\mathrm{d}}{\mathrm{d} t} E(u)=\left(\frac{\delta E(u)}{\delta u}, \frac{\partial u}{\partial t}\right)=-\left\|{\partial_{t} u}\right\|^2 \leq 0, \quad \forall t>0,
\end{equation}
where $(\cdot, \cdot)$ and $\|\cdot\|$ represent the standard $L^2$ inner product and norm.
% respectively, defined as
% \begin{equation}\label{eq:InnerProduction} 
%     (u, v)=\int_{\Omega} u v \mathrm{~d} x, \quad\|u\|=\left(u,u\right)^{\frac{1}{2}}, \quad \forall u, v \in L^2(\Omega).
% \end{equation}
Given the absence of exact solutions for many phase field models including the Allen--Cahn equations, the acquisition of precise and stable numerical simulations that faithfully replicate their physical characteristics becomes essential. To reduce the risk of encountering nonphysical results, it is important to design accurate numerical methods that preserve the energy dissipation law and the MBP.

In recent years, there has been a significant focus on 
the development and analysis of time stepping schemes that preserve the MBP of Allen--Cahn equations as well as the energy dissipation law of general gradient flow models.
These efforts have explored a wide range of methods, including convex splitting methods \cite{du91convex,Eyre1998convex,wise2009convexspliting,guan2014convexsplitting}, operator splitting methods \cite{cheng2015fast,li2017convergence,LI-QUAN-XU2022splitting1,LI-QUAN-XU2022splitting2}, stabilized implicit-explicit (IMEX) schemes \cite{XU-TANG2006IMEX,TANG-YANG2016IMEX,SHEN-TANG-YANG2016IMEX_GeneralAC,FU-TANG-YANG2022IMEX_Energy,liao2023BDF_IMEX}, 
integrating factor Runge--Kutta (IFRK) methods \cite{JU-LI-QIAO-YANG2021IFRK,LI-LI-JU-FENG2021sIFRK}, exponential time differencing (ETD) schemes \cite{COX-MATTHEWS2002ETD,DU-JU-LI-QIAO2019ETD,DU-JU-LI-QIAO2021ETD,FU-Yang2022ETDRK2Energy,FU-SHEN-YANG2024HigherOrder_ED}, invariant energy quadratization (IEQ) schemes \cite{XiaofengYang2016IEQ,XU-YANG-ZHANG-XIE2019S-IEQ,YANG-ZHANG2020IEQ}, scalar auxiliary variable (SAV) schemes \cite{SHEN-XU-YANG2018SAV,SHEN-XU-YANG2019SAV,Akrivis-LI-LI2019RKSAV,huang2020highly}, Lagrange multiplier approach \cite{cheng2020LagrangeMultiplier,cheng2022new}.
However, among all these aforementioned methods, most methods preserving the energy dissipation law, use some modified forms of energy, which differ from the original definition found in the continuous partial differential equation setting. For example,  the IEQ/SAV method discretizes a reformulation of the gradient flow equation by introducing auxiliary variable to ensure the modified energy dissipation \cite{XiaofengYang2016IEQ,SHEN-XU-YANG2018SAV,Akrivis-LI-LI2019RKSAV}. Moreover, the modified energy of high-order IMEX backward differentiation formula method can be constructed \cite{liao2023BDF_IMEX} based on the Nevanlinna--Odeh  multiplier technique \cite{lubich2013backward, akrivis2015IEBDFstability, akrivis2015fully}. See also related topics such as phase field method for geometric moving interface \cite{feng2003numerical, du-feng2020phase}, the energy stability analysis for nonuniform time steps \cite{LIAO-TANG-ZHOU2020BDF,akrivis2024variable}, and the error estimates with only polynomial dependence on $\varepsilon^{-1}$ \cite{feng2004error, kovacs2017numerical, Akrivis22ErrorBDF, harder2022error} etc.

Constructing high-order schemes that preserve the original energy dissipation law is a significant and intriguing challenge. 
In \cite{Lubich2014energy}, 
the energy dissipation property of standard implicit Runge–Kutta methods was demonstrated for gradient systems with Lipschitz nonlinearity, provided that the time step was sufficiently small.
However, these schemes require to solve a nonlinear equation at each time step. Furthermore, while it is known that the time step size must be small, the exact limitations have not been explicitly defined. 
%More  recently, the ETDRK method \cite{FU-Yang2022ETDRK2Energy,FU-SHEN-YANG2024HigherOrder_ED} and the IMEX-RK method \cite{FU-TANG-YANG2022IMEX_Energy} have emerged as linear, high-order schemes capable of preserving the original energy. {\color{red} no time step restriction???}
%Despite these advancements, the highest order achieved by these methods is third order.
More recently, exponential time differencing Runge--Kutta (ETDRK) schemes have demonstrated the capability to dissipate the original energy. 
%Both the first-order ETDRK (ETDRK1) and second-order ETDRK (ETDRK2) schemes have been shown to unconditionally preserve the MBP for a class of semi-linear parabolic equations in \cite{DU-JU-LI-QIAO2021ETD}. Moreover, 
For example, in \cite{FU-Yang2022ETDRK2Energy}, the ETDRK1 and ETDRK2 schemes have
been shown to unconditionally preserves the original energy dissipation law.
Then the work was extended to high-order ETDRK schemes in \cite{FU-SHEN-YANG2024HigherOrder_ED}, where some positive definiteness conditions are given to ensure the original energy dissipation.
Based on these conditions, the authors discover some third-order ETDRK schemes which dissipate the original energy. However, since these conditions are not easy to meet for higher-order schemes, the existence of qualified ETDRK4 schemes remains unclear \cite{FU-SHEN-YANG2024HigherOrder_ED}.
In addition, the proof of the original energy dissipation law in \cite{FU-SHEN-YANG2024HigherOrder_ED} requires a Lipschitz condition assumption on the nonlinear term $f$. For the Allen--Cahn equations, this assumption will be automatically satisfied if the MBP is preserved. However, only the ETDRK1 and ETDRK2 schemes have been proven to preserve the MBP unconditionally \cite{DU-JU-LI-QIAO2021ETD}.
Numerical observations have further suggested that third-order and higher-order ETDRK schemes do not preserve the MBP unconditionally. See also \cite{FU-TANG-YANG2022IMEX_Energy} for related discussion for IMEX-RK schemes up to third-order.
Designing arbitrarily high-order, unconditionally MBP-preserving, and energy-dissipative time stepping schemes remains a challenging task.
Some advances have been made in \cite{LI-YANG-ZHOU2020cutoff} by applying a cut-off post-processing technique to guarantee the MBP, and then in \cite{YANG-YUAN-ZHOU2022cutoff}, by combining the high-order SAV Runge--Kutta method \cite{Akrivis-LI-LI2019RKSAV} to decrease the modified energy.

% Although exponential time differencing Runge--Kutta (ETDRK) methods have shown the capability to decrease the original form of energy unconditionally \cite{FU-Yang2022ETDRK2-Energy,FU-SHEN-YANG2024HigherOrder_EnergyDecay}, their accuracy is currently limited to third-order and the third-order ETDRK scheme does not inherently preserve the MBP. 
% As for the MBP, it has been established that only the lower-order ETDRK schemes, specifically the ETDRK1 and ETDRK2, can unconditionally preserve the MBP, as demonstrated for a class of parabolic equations in \cite{DU2021ETD}. Numerical observations have further confirmed that third-order and higher-order ETDRK schemes do not preserve the MBP. Designing arbitrarily high-order unconditionally MBP-preserving ETDRK methods remains a challenge.
% Another drawback of the ETDRK schemes is that proving the energy dissipation law for ETDRK schemes requires assuming a Lipschitz condition for $f$. While the Lipschitz condition can be automatically satisfied by the MBP, high-order ETDRK schemes fail to preserve the MBP.

In this paper, we focus on \cref{eq:AC} with a general nonlinear term  $f: \mathbb{R} \rightarrow \mathbb{R}$ given by a continuously differentiable function satisfying:
\begin{equation}\label{eq:f(beta)<0}
    \exists \text{ a constant } \beta>0, \ \text{ such that } f(\beta) \le 0 \le f(-\beta).
\end{equation}
Equipped with homogeneous Neumann boundary condition, the MBP holds \cite{DU-JU-LI-QIAO2021ETD} in the sense that if the absolute value of the initial value is bounded by $\beta$, then the absolute value of the solution is also bounded by $\beta$ for all time, i.e.,
\begin{equation}\label{eq:MBP_continuous}
    \left\|u^{0}\right\|_{\infty} \le \beta \quad \Longrightarrow \quad \left\|u(t,x)\right\|_{\infty}\leq \beta, 
    \quad \forall t >0,
\end{equation}
where the maximum norm $\left\|\cdot\right\|_{\infty}$ is defined as
$
    \left\|u\right\|_{\infty} \coloneq \max_{x \in \bar{\Omega}} \left|u(x)\right| $ for any $ u \in C(\bar{\Omega}). 
$
The energy dissipation law is also satisfied with respect to the energy \cref{eq:E} with $F$ being a smooth potential function satisfying $F^{\prime}=-f$. 

%Based on previous research in \cite{DU-JU-LI-QIAO2019ETD,DU-JU-LI-QIAO2021ETD}, which pointed out that the failure of MBP arises from unboundedness of the interpolation polynomial, 
%we aim to first establish the original energy dissipation law of arbitrarily high-order (in terms of time) ETDRK schemes, and then to design arbitrarily high-order, MBP and energy dissipation preserving ETDRK schemes for the Allen--Cahn type gradient flows. 
Under the assumption that the nonlinearity $f$ is Lipschitz continuous on $\mathbb R$, we first prove that an arbitrary high-order ETDRK method will preserve the original energy dissipation law if the time step size $\tau$ is smaller than some constant $\tau_{\rm max}$ and the interpolation nodes are located on the interval $[0,\tau]$. 
However, it is known that the third-order ETDRK method of Allen--Cahn equation does not preserve the MBP and consequently the Lipschitz constant for $f$ can not be obtained explicitly.
We then propose a rescaling technique for ETDRK schemes to preserve MBP unconditionally, where the interpolation polynomial is adjusted slightly without compromising the convergence order.
We also prove the original energy dissipation law of these rescaled ETDRK methods for small time step, without assuming the Lipschitz continuity of $f$ on $\mathbb R$.
The rigorous convergence analysis of arbitrarily high-order rescaled ETDRK method is provided. 
To the best of our knowledge, this is the first work on arbitrarily high-order, MBP and energy dissipation preserving ETDRK schemes for the Allen--Cahn type gradient flows.
%to establish the original energy dissipation of an arbitrarily high-order ETDRK method for Allen--Cahn type gradient flows. %which unconditionally preserve the MBP and conditionally preserve the original energy dissipation law.

 The rest of this paper is organized as follows. In Section \ref{section:original energy}, we first introduce a class of arbitrarily high-order ETDRK schemes and then prove their preservation of the energy dissipation law under specific time step size restrictions. 
 In Section \ref{section:rescaling technique}, we introduce a rescaling technique that allows these schemes to preserve the MBP unconditionally and the energy dissipation law with small time steps, followed by analysis of temporal convergence.
 Numerical experiments are carried out to validate the theoretical results and demonstrate the performance of the proposed schemes in Section \ref{section:numerical experiments}. Finally, some concluding remarks are given in Section \ref{section:conclusions}.

\section{Original energy dissipation of ETDRK methods}\label{section:original energy}
In this section, we first introduce a class of  arbitrarily high-order ETDRK methods for solving the Allen--Cahn equation \cref{eq:AC}, following the abstract framework outlined in \cite{DU-JU-LI-QIAO2021ETD}. Then, we prove that the original energy decreases under a certain restriction of time-step size.
% \textcolor{red}{For gradient systems, Lubich et. al. \cite{Lubich2014energy} have showed that algebraically stable Runge--Kutta methods can decrease the original energy within a mild time-step size restriction. 
% We conjecture that similar properties hold for arbitrarily high-order ETD Runge--Kutta schemes as well.}

Initially, we establish the original energy dissipation law under the assumption of the Lipschitz continuity of $f$. Subsequently, in Section \ref{section:rescaling technique}, we apply the MBP to relax this assumption.
Assume that $f$ satisfies the Lipschitz condition with a Lipschitz constant $C_l$, i.e.,
\begin{equation}\label{eq:Lipschitz}
|f(u)-f(v)| \le C_l |u-v| \quad \forall u, v \in \mathbb{R}.
\end{equation}
Following this, we introduce a stabilizing constant $\kappa$, satisfying: 
\begin{equation}\label{eq:kappa cond}
\kappa \ge C_l.
\end{equation}
By adding and subtracting a stabilization term $\kappa {u}$ to the Allen--Cahn equation \cref{eq:AC}, we derive an equivalent form of \cref{eq:AC}
\begin{equation}\label{eq:AC stable form}
    u_{t}=\mathcal{L}_{\kappa} u+\mathcal{N}(u), \quad {x} \in \Omega,~ t>0,
\end{equation}
where the linear operator $\mathcal L_\kappa$ and nonlinear operator $\mathcal N$ are defined as
\begin{equation}\label{eq:L_kappa, N}
        \mathcal{L}_{\kappa}\coloneqq \varepsilon^2 \Delta - \kappa \mathcal{I},\quad
    \mathcal{N} \coloneqq f+\kappa \mathcal{I}, 
\end{equation}
and $\mathcal{I}$ denotes the identity operator. 

% It is easy to know that 
% \begin{equation}\label{eq:Npri_positive}
%     0 \le \mathcal{N}^{\prime}(\xi) \le 2\kappa
% \end{equation} for any $\xi \in[-\beta, \beta]$.
% The researches in \cite{DU-JU-LI-QIAO2019ETD,DU-JU-LI-QIAO2021ETD} have shown that the requirement of $\kappa \ge  \max _{|\xi| \le \beta}\left|f^{\prime}(\xi)\right|$ is crucial to preserve the MBP for ETDRK schemes. 

Given a positive integer $N$, let the time interval $[0,T]$ be divided into $N$ subintervals with a uniform time step $\tau=T/N$, and define $t_n=n \tau, n=0,1,\cdots,N$.
To solve the Allen--Cahn equation \cref{eq:AC}, we focus on the equivalent equation \cref{eq:AC stable form} over the interval $\left[t_n, t_{n+1}\right]$, or equivalently $w^{n}(x,s)=u\left(x,t_n+s\right)$  satisfying the system
\begin{equation} \label{eq:w(s)}
    \left\{
    \begin{aligned}
        &\partial_{s} w^{n}  =\mathcal{L}_{\kappa} w^{n}+\mathcal{N}(w^{n}), &  & {x} \in \Omega ,~ s \in(0, \tau], \\
       & w^{n}(x,0)    =u\left(x,t_{n}\right),                         &  & {x} \in \bar\Omega ,
    \end{aligned}
    \right.
\end{equation}
equipped with homogeneous Neumann boundary condition.
The key idea of ETDRK is applying Duhamel's principle to this system 
to deduce
\begin{equation}\label{eq:Duhamel}
    w^{n}(x, \tau)= \mathrm{e}^{\tau\mathcal{L}_\kappa} w^{n}(0, x)+ \int_0^\tau \mathrm{e}^{(\tau-s)\mathcal{L}_\kappa} \mathcal{N}[w^{n}(x,s)] \mathrm{~d} s
\end{equation}
and then approximating the nonlinear function $\mathcal{N}[u(t_n+s)]$ in the integral. For instance, one straightforward approach is to set $\mathcal{N}[u(t_n+s)] \approx \mathcal{N}[u(t_n)]$, which introduces a truncation error of $O(\tau)$. This approximation leads to a first-order scheme, the ETDRK1 scheme, i.e., for $n \ge 0$,
\begin{equation}\label{eq:ETDRK1}
    u^{n+1}=\mathrm{e}^{\tau\mathcal{L}_\kappa} u^{n}+ \left( \mathrm{e}^{\tau\mathcal{L}_\kappa}-\mathcal{I} \right)\mathcal{L}_\kappa^{-1} \mathcal{N}(u^{n}),
\end{equation}
where $u^n$ is the numerical solution approximating the exact solution at $t = t_n$, and $u^0$ is the initial condition given in \cref{eq:AC}.

Advancing beyond the basic ETDRK1 scheme, one can derive higher-order ETDRK schemes by employing  interpolation polynomial to estimate the nonlinear term $\mathcal{N}[u(t_n+s)]$ for $s\in[0,\tau]$. 
For any integer $r \geq 0$, we construct the $(r+1)$th-order ETDRK scheme by selecting $r+1$ nodes $\left\{0=a_{r, 0}<a_{r, 1}<\cdots<a_{r, r} \leq 1\right\}$ within the interval $[0,1]$. 
We then interpolate the function $\mathcal{N}\left[u\left(t_n+s\right)\right]$ at the times $\left\{a_{r, k} \tau\right\}_{k=0}^r$ to form a polynomial $P_r^n(s)$ of degree $r$, which results in a truncation error of $O\left(\tau^{r+1}\right)$.
Then, we approximate $\mathcal{N}\left[u\left(t_n+s\right)\right]$ by $P^{n}_r(s)$ to obtain the following $(r+1)$th-order ETDRK scheme, i.e., compute $u^{n+1}=w_{r+1}^n(\tau)$ by solving the following linear partial differential equation
\begin{equation}\label{eq:ETDRKr}
    \left\{
    \begin{aligned}
       & \partial_{s} w_{r+1}^{n}   =\mathcal{L}_{\kappa} w_{r+1}^{n}+P^{n}_{r}(s), &  & {x} \in \Omega , ~ s \in(0, \tau], \\
            & w_{r+1}^{n}(x,0)      =u^n({x}),                                  &  & {x} \in \bar{\Omega},
    \end{aligned}
    \right.
\end{equation}
equipped with homogeneous Neumann boundary condition. 
Here, $u^n$ is the numerical solution of the ETDRK$(r+1)$ scheme at $t_n$ and $u^0$ is the initial condition given in \cref{eq:AC}. 
More precisely, the polynomial $P_{r}(s)$ can be written as
\begin{equation}\label{eq:Pr}
  P^{n}_{r}(s) = \mathcal{N}(u^{n}) + c_{r,1}\frac{s}{\tau}  + c_{r,2}\left(\frac{s}{\tau}\right) ^{2}+ \cdots + c_{r,r}\left(\frac{s}{\tau}\right) ^{r},
\end{equation}
where the coefficients $\left\{c_{r,k}\right\}_{k=1}^{r}$ is determined by
\begin{equation}\label{eq:vandermonde V_r}
    \begin{pmatrix}
        \left( a_{r,1} \right) ^1&		\left( a_{r,1} \right) ^2&		\cdots&		\left( a_{r,1} \right) ^{r}\\
        \left( a_{r,2} \right) ^1&		\left( a_{r,2} \right) ^2&		\cdots&		\left( a_{r,2} \right) ^{r}\\
        \vdots&		\vdots&		\ddots&		\vdots\\
        \left( a_{r,r} \right) ^1&		\left( a_{r,r} \right) ^2&		\cdots&		\left( a_{r,r} \right) ^{r}\\
    \end{pmatrix} 
    \begin{pmatrix}
        c^{n}_{r,1} \\
        c^{n}_{r,2} \\
        \vdots \\
        c^{n}_{r,r}
    \end{pmatrix}=
    \begin{pmatrix}
        \mathcal{N}\left[w^{n}_{r}(a_{r,1}\tau)\right] -\mathcal{N}(u^{n}) \\
        \mathcal{N}\left[w^{n}_{r}(a_{r,2}\tau)\right] -\mathcal{N}(u^{n})\\
        \vdots \\
        \mathcal{N}\left[w^{n}_{r}(a_{r,r}\tau)\right] -\mathcal{N}(u^{n})
    \end{pmatrix}.
\end{equation}
Here, $w^{n}_{r}(s)$ is the solution of the ETDRK$r$ scheme, which means that we construct the high-order ETDRK schemes in an iterative way.
To simplify notation, in \cref{eq:vandermonde V_r}, let $V_r$ represent the Vandermonde matrix, and $\boldsymbol{c}^{n}_{r}$ and $\boldsymbol{d}^{n}_{r}$ denote the coefficients vector and right-hand side vector, respectively. Thus, \cref{eq:vandermonde V_r} can be written as $V_r\boldsymbol{c}^{n}_{r} = \boldsymbol{d}^{n}_{r}$.

Up to now, we have only presented the differential forms of ETDRK$(r+1)$ schemes, but we also need explicit formulas which can be directly implemented for computations.
First, for any integer $j\ge 0$, we  
have the integration
\begin{equation}\label{eq:s^k e(t-s)L}
  \begin{aligned}
    \int_{0}^{t}  \mathrm{e}^{(t-s)\mathcal{L} _{\kappa}} s^j \mathrm{~d}s 
    &= j!\left(\mathrm{e}^{t\mathcal{L} _{\kappa}}-\sum_{k=0}^j{\frac{1}{k!}(t\mathcal{L} _{\kappa})^k}\right)\mathcal{L}_{\kappa}^{-j-1} = j!t^{j+1} \phi_{j+1}(t\mathcal{L}_{\kappa}) , 
  \end{aligned}
\end{equation}
where 
\begin{equation}\label{eq:psi_k def}
    \phi_{j+1}(z) \coloneqq  \left( \mathrm{e}^{z}- \sum_{k=0}^{j}\frac{z^k}{k!} \right)z^{-(j+1)}.
\end{equation}
Then, applying Duhamel's principle for the system \cref{eq:ETDRKr}, we have
\begin{equation}\label{eq:explicit_fomula}
    \begin{aligned}
        w^{n}_{r+1}(s)  
        & = \mathrm{e}^{s\mathcal{L}_{\kappa}}u^n + \int_{0}^{s} \mathrm{e}^{(s-\sigma)\mathcal{L} _{\kappa}} P_{r}(\sigma) \mathrm{~d}\sigma  \\
        & = \mathrm{e}^{s\mathcal{L}_{\kappa}}u^n +  \sum_{j=0}^{r}j!s^{j+1}\phi_{j+1}\left(s\mathcal{L}_{\kappa}\right) \frac{c_{r,j}}{\tau^j} \\
        & = \mathrm{e}^{s\mathcal{L}_{\kappa}}u^n + \left( \mathrm{e}^{s\mathcal{L}_\kappa}-\mathcal{I} \right)\mathcal{L}_\kappa^{-1} \mathcal{N}(u^{n}) + \tau \sum_{j=1}^{r}j! \left(\frac{s}{\tau} \right)^{j+1} \phi_{j+1} \\
        & = w_{1}^{n}(s) + \tau \sum_{j=1}^{r}j! \left(\frac{s}{\tau} \right)^{j+1} \phi_{j+1}
 \left(s\mathcal{L}_{\kappa}\right)c_{r,j}^{n}.
    \end{aligned}
\end{equation}
Therefore, letting $s = \tau$, we have the explicit formula of the ETDRK$(r+1)$ scheme:
\begin{equation}\label{eq:explicit_fomula_tau}
    u^{n+1} \coloneqq  w^{n}_{r+1}(\tau) = w^{n}_{1}(\tau) + \tau \sum_{j=1}^{r}j!\phi_{j+1}\left(\tau\mathcal{L}_{\kappa}\right)c_{r,j}^{n}.
\end{equation}

\begin{lemma}\label{lemma:psi_k bound}
For any $t>0$, consider the negative-definite operator $t\mathcal{L}_{\kappa}=t(\varepsilon^{2}\Delta-\kappa \mathcal{I})$ with homogeneous Neumann boundary condition.
For any integer $k \ge 0$ and $\lambda \in (0,1)$, the following inequality holds:
    \begin{equation}\label{eq:psi_k bound}
    \left\|\lambda^{k}\phi_k\left(\lambda t\mathcal{L}_{\kappa}\right)v\right\| \le  \left\|\phi_{k}(t\mathcal{L} _{\kappa})v\right\| \le \frac{1}{k!} \left\|v\right\|,  
    \quad \forall v \in L^{2}(\Omega),
    \end{equation}
    where $\|\cdot\|$ is short for $\|\cdot\|_{L^2(\Omega)}$.
\end{lemma}
\begin{proof}
    By the Lagrange remainder of Taylor expansion of $e^x$, we have
\begin{equation}\label{eq:temp_1}
    0<\phi_k(x)
    =\Big(\mathrm{e}^{x}- \sum_{j=0}^{k-1}\frac{x^j}{j!} \Big)x^{-k}
    =\frac{\mathrm{e}^{\xi }}{k!}
    < \frac{1}{k!},\quad \forall x<0,
\end{equation}
where $\xi \in (x,0)$. Let $(\mu_j,\varphi_j)_{j=1}^\infty$ be eigenpairs of the selfadjoint and positive definite operator $\mathcal{L}_\kappa$, and note that $\varphi_j$ forms a complete orthogonal basis of $L^2(\Omega)$. As a result, for any $v\in L^2(\Omega)$, there holds
\begin{equation}\label{eq:temp_11}
\left\|\phi_{k}(t\mathcal{L} _{\kappa})v\right\|^2
\le \sum_{j=1}^\infty \phi_{k}(t\mu_j)^2 |(v,\varphi_j)|^2
\le \frac{1}{k!} \sum_{j=1}^\infty |(v,\varphi_j)|^2
= \frac{1}{k!} \| v \|^2.
\end{equation}

%By the periodic boundary condition, we have
%\begin{equation}\label{eq:temp_1}
%    \left(\frac{1}{k!} \left\|v\right\|\right) ^{2}-\left\|\phi_{k}(t\mathcal{L} _{\kappa})v\right\|^{2}  = \left(v,\left[\left(\frac{1}{k!}\right)^{2} - \left(\phi_{k}(t\mathcal{L} _{\kappa})\right)^{2}  \right]v\right). 
%\end{equation}
%Since $t\mathcal{L}_{\kappa}=t(\varepsilon^{2}\Delta-\kappa \mathcal{I})$ is a negative-definite operator and $g(x)\coloneqq \left(\frac{1}{k!}\right)^{2}  - \left(\phi_k(x)\right)^{2} $ are analytic functions defined on the set of eigenvalues of $t\mathcal{L} _{\kappa}$, the eigenvalues of $g(t\mathcal{L}_{\kappa})$ are $\left\{g\left(\lambda_i\right)\right\}_{i \in \mathcal{M}}$, where $\left\{\lambda_i\right\}_{i \in \mathcal{M}}$ are the eigenvalues of $t\mathcal{L}_{\kappa}$. It is easy to show that $g(t\mathcal{L}_{\kappa})$ are positive-definite operators because $g(x)>0$ for any $x<0$. Therefore, the second inequality of \cref{eq:psi_k bound} arises from \cref{eq:temp_1}.

Next, we consider the first inequality of \cref{eq:psi_k bound}. For any integer $k\ge 0$, let us define a function $h_k$ as follows:
\begin{equation}\label{eq:h_k}
    h_k(\lambda,x) \coloneq \lambda^{k}\phi_k(\lambda x) = \Big(\mathrm{e}^{\lambda x}- \sum_{j=0}^{k-1}\frac{(\lambda x)^j}{j!} \Big)x^{-k} , \quad \lambda >0.
\end{equation}
It is easy to check that
\[
\frac{\partial h_k}{\partial \lambda} = \Big( \mathrm{e}^{\lambda x}- \sum_{j=0}^{k-2}\frac{(\lambda x)^j}{j!} \Big)x^{-(k-1)} = \frac{e^{\zeta}}{(k-1)!} \lambda^{k-1}  >0, \quad \forall \lambda>0,  x \in \mathbb{R},
\]
where $\zeta$ is between 0 and $\lambda x$.
Thus, we obtain 
\begin{equation}\label{eq:tempa11}
    0 = h_k(0,x) < h_k\left(\lambda,x\right) < h_k(1,x)=\phi_k(x)
\end{equation}
for any $x \in \mathbb{R}$. This combined with the argument in \cref{eq:temp_11} immediately leads to the 
desired estimate $\Vert \phi _{k}( t\mathcal{L}_{\kappa }) v\Vert ^{2} - \left\Vert \lambda ^{k} \phi _{k}( \lambda t\mathcal{L}_{\kappa }) v\right\Vert ^{2}  \ge 0$.
%Therefore, we have
%\begin{equation}\label{eq:temp_2}
%\begin{aligned}
%\Vert \phi _{k}( t\mathcal{L}_{\kappa }) v\Vert ^{2} - \left\Vert \lambda ^{k} \phi _{k}( \lambda t\mathcal{L}_{\kappa }) v\right\Vert ^{2}  
%& =\Vert h_{k}( 1,t\mathcal{L}_{\kappa }) v\Vert ^{2} - \Vert h_{k}( \lambda , t\mathcal{L}_{\kappa }) v\Vert ^{2} \\
% & = \left(v, \left[(h_{k}( 1,t\mathcal{L}_{\kappa }))^{2}- (h_{k}( \lambda ,\lambda t\mathcal{L}_{\kappa }))^{2}\right]  v\right) .
%\end{aligned}
%\end{equation}
%By \cref{eq:tempa11}, we have $\Vert \phi _{k}( t\mathcal{L}_{\kappa }) v\Vert ^{2} - \left\Vert \lambda ^{k} \phi _{k}( \lambda t\mathcal{L}_{\kappa }) v\right\Vert ^{2}  \ge 0$.
\end{proof}

% \begin{lemma}\label{lemma:norm of sum}
% Suppose  $m$ is a positive integer, the following inequality
% \[
% \left\|\sum_{i=1}^{m}v_i\right\|^{2}\le m \sum_{i=1}^{m} \left\|v_i\right\|^{2}
% \]
% holds for any $v \in L^{2}(\Omega)$.
% \end{lemma}
% \begin{proof}
%     In fact, it is a version of the Cauchy--Schwarz inequality. More precisely, we have
%     \begin{align*}
%         \left\|\sum_{i=1}^{m}v_i\right\|^{2}
%         &= \left(\sum_{i=1}^{m}v_i,\sum_{i=1}^{m}v_i\right) \\
%         &= \sum_{i=1}^{m}\sum_{j=1}^{m}\left(v_i,v_j\right) \\
%         &\le  \sum_{i=1}^{m}\sum_{j=1}^{m} \left\|v_i\right\|\left\|v_j\right\|\\
%         &\le  \sum_{i=1}^{m}\sum_{j=1}^{m}\frac{1}{2} \left(\left\|v_i\right\|^{2}+\left\|v_j\right\|^{2}\right) \\
%         &=m \sum_{i=1}^{m} \left\|v_i\right\|^{2}.
%     \end{align*}
% \end{proof}

\begin{theorem}\label{theorem:energy decay main theorem}
Suppose that \cref{eq:f(beta)<0,eq:Lipschitz,eq:kappa cond} hold. 
Then the ETDRK$r$ scheme preserves the original energy dissipation law within a certain time-step size restriction. More precisely, for any integer $r\ge 1$, there exists a positive constant $\tau_{\max,r}$ independent of $\varepsilon$, such that
the solution $\left\{u^{n}\right\}_{n\ge 0} $ to the ETDRK$r$ scheme satisfies
    \[
        E\left(u^{n+1}\right) - E\left(u^{n}\right) \coloneqq E\left(w^{n}_{r}(\tau)\right) - E\left(w^{n}_{r}(0)\right) \le 0 , 
    \]
for all $\tau\le \tau_{\max,r}$. The time-step size restriction is  
\[
\tau_{\max,1} = + \infty, \quad \tau_{\max,r} = \frac{1}{4\kappa} \min \left\{{\sigma_{\min}\left(V_1\right)},\frac{\sigma_{\min}\left(V_2\right)}{2}, \ldots ,\frac{\sigma_{\min}\left(V_{r-1}\right)}{r-1}\right\}, \quad r\ge 2,
\]
where $\sigma_{\min}(A)$ represents the minimum singular value of a matrix $A$, and $V_{k}$ is the $k \times k$ Vandermonde matrix of interpolation nodes defined in \cref{eq:vandermonde V_r}.
\end{theorem}
\begin{proof}
We divide our proof into two parts. In the first part, we present an inequality that is a sufficient condition for the original energy dissipation law. In the second part, we use mathematical induction to establish this inequality, thereby obtaining a time step restriction to preserve the energy dissipation for arbitrarily high-order ETDRK methods.

{\em Part I.} We propose a sufficient condition for the original energy dissipation law and set it as the objective for mathematical induction.
From \cref{eq:E} and \cref{eq:dE/dt}, we derive the following expression for the derivative of the original energy of $w_{r}^{n}$ in \cref{eq:ETDRKr}:
\begin{equation}\label{eq:dE(w)/dt}
        \frac{\mathrm{d}}{\mathrm{d} s} E(w_{r}^{n} )=\left(\frac{\delta E(w_{r}^{n} )}{\delta w_{r}^{n}} ,\frac{\partial w_{r}^{n}}{\partial s}\right)
        =\left( -\mathcal{L}_{\kappa } w_{r}^{n} -\mathcal{N}\left( w_{r}^{n}\right) ,\frac{\partial w_{r}^{n}}{\partial s}\right).
\end{equation}
Then, the difference of original energy between two adjacent moments can be expressed as 
\begin{equation}\label{eq:E_2-E_1}
\begin{aligned}
    E\left( u^{n+1}\right) -E\left( u^{n}\right) 
    &= \int_{0}^{\tau} \left( -\mathcal{L}_{\kappa } w_{r}^{n} -\mathcal{N}\left( w_{r}^{n}\right) ,\frac{\partial w_{r}^{n}}{\partial s}\right) \mathrm{~d}s \\
    &=\left(\int _{0}^{\tau }\left(-\mathcal{L}_{\kappa } w^{n}_{r} -\mathcal{N}( w^{n}_{r})\right)   \partial_s w_{r}^{n} \mathrm{~d} s,1\right). 
\end{aligned}
\end{equation}
For any $x \in \Omega$, we have
\begin{equation}\label{eq:int_N}
\begin{aligned}
    \int _{0}^{\tau }\mathcal{N} (w_{r}^{n} )\partial_s w_{r}^{n} \mathrm{~d} s 
    &= \left[-F( u^{n+1}) +\frac{\kappa }{2} \left(u^{n+1}\right) ^{2}\right] -\left[-F( u^{n}) +\frac{\kappa }{2} \left(u^{n}\right) ^{2}\right].
\end{aligned}
\end{equation}
By the Lipschitz condition assumption \cref{eq:Lipschitz} of $f=-F^{\prime}$ and the condition \cref{eq:kappa cond},
% Since the nonlinear function $f = - F^{^{\prime}}$ is Lipschitz continuous {\color{red}???}, 
we have
\begin{equation}\label{eq:int_N_1}
    \left[-F( u^{n+1}) \right] -\left[-F( u^{n}) \right] \ge f(u^{n})\left(u^{n+1}-u^{n}\right)-\frac{\kappa}{2}\left(u^{n+1}-u^{n}\right)^{2}.  
\end{equation}
Substituting \cref{eq:int_N_1} into \cref{eq:int_N}, we have
\begin{equation}\label{eq:nonamea}
    \begin{aligned}
        \int _{0}^{\tau }\mathcal{N} (w_{r}^{n} )\partial_s w_{r}^{n} \mathrm{~d} s  \ge f(u^{n})\left(u^{n+1}-u^{n}\right)+{\kappa}u^{n}\left(u^{n+1}-u^{n}\right)  
      = \mathcal{N}\left( u^{n}\right)\left( u^{n+1} -u^{n}\right).
    \end{aligned}
\end{equation}
Substituting \cref{eq:nonamea} into \cref{eq:E_2-E_1}, and by the differential form of ETDRK$r$ scheme \cref{eq:ETDRKr}, i.e. $\partial_{s} w_{r}^{n} =\mathcal{L}_{\kappa} w_{r}^{n}+P^{n}_{r-1}(s)$, 
we have
    \begin{align}\label{eq:E_2-E_1 process}
  E\left( u^{n+1}\right) -E\left( u^{n}\right)  
    \notag &\le \left(\int _{0}^{\tau }\left(-\partial_s  w _{r}^{n} +P_{r-1}^{n} (s) -\mathcal{N}(u^{n}) \right)  \partial_s w_{r}^{n} \mathrm{~d} s ,1\right) \\
    \notag &= \left( - \int _{0}^{\tau }\left(\partial_s  w _{r}^{n}\right)^{2}  \mathrm{~d} s + \int _{0}^{\tau }\left(P_{r-1}^{n} (s) -\mathcal{N}(u^{n}) \right)  \partial_s w_{r}^{n} \mathrm{~d} s ,1\right) \\
    \notag &\le  \left( -\frac{1}{2} \int _{0}^{\tau}\left(\partial_s  w _{r}^{n}\right)^{2} \mathrm{~d} s + \frac{1}{2}\int _{0}^{\tau}\left(P_{r-1}^{n} (s) -\mathcal{N}(u^{n}) \right)^{2} \mathrm{~d} s ,1\right)\\
     &= -\frac{1}{2\tau} \left(  \int _{0}^{\tau}1\mathrm{~d} s \int _{0}^{\tau}\left(\partial_s  w _{r}^{n}\right)^{2} \mathrm{~d} s ,1\right) +  \frac{1}{2}\left(\int _{0}^{\tau}\left(P_{r-1}^{n} (s) -\mathcal{N}(u^{n}) \right)^{2} \mathrm{~d} s ,1\right)\\
%    &\le  -\frac{1}{2\tau} \left(  \left( \int _{0}^{\tau} \left|\partial_s  w_{r}^{n}\right|  \mathrm{~d} s \right)^{2} ,1\right) +  \frac{1}{2}\left(\int _{0}^{\tau}\left(P_{r-1}^{n} (s) -\mathcal{N}(u^{n}) \right)^{2} \mathrm{~d} s ,1\right)\\
    \notag &\le  -\frac{1}{2\tau} \left(  \left( \int _{0}^{\tau} \partial_s  w_{r}^{n} \mathrm{~d} s \right)^{2} ,1\right) +  \frac{1}{2}\left(\int _{0}^{\tau}\left(P_{r-1}^{n} (s) -\mathcal{N}(u^{n}) \right)^{2} \mathrm{~d} s ,1\right)\\
    \notag &= -\frac{1}{2\tau} \left\|u^{n+1} -u^{n}\right\|^{2} + \frac{1}{2}\left(\int _{0}^{\tau }\left( \sum _{k=1}^{r-1} \left({s}/{\tau}\right)^{k} c_{r-1,k}^{n} \right)^{2}\mathrm{~d}s ,1\right)  \\
    \notag &\le -\frac{1}{2\tau} \left\|u^{n+1} -u^{n}\right\|^{2}+\frac{1}{2} {\tau(r-1)}\left(\sum_{k=1}^{r-1}\left(c_{r-1,k}^{n}\right)^{2} ,1\right)  .
    \end{align}
To simplify notations, we use the notation which was defined in \cref{eq:vandermonde V_r}, i.e.
\[
    \boldsymbol{c}^{n}_{r-1} = \left( c_{r-1,1}^{n},c_{r-1,2}^{n}, \ldots ,c_{r-1,r-1}^{n} \right) ^{\mathsf{T}},
\]
and define the 2-norm of vectors as 
\[
    \left|\boldsymbol{c}^{n}_{r-1}\right|  \coloneq \left( \sum_{k=1}^{r-1}\left(c_{r-1,k}^{n}\right)^{2} \right) ^{\frac{1}{2}}. 
\]
Then, we have
\begin{equation}\label{eq:Aim}
    E\left(u^{n+1}\right) -E\left(u^{n}\right) \le
    \frac{1}{2\tau}\left[{\tau^{2}(r-1)}\left\|\left|\boldsymbol{c}^{n}_{r-1}\right|\right\|^{2} - \left\| u^{n+1} -u^{n}\right\|^{2}  \right].
\end{equation}
To prove $E\left(u^{n+1}\right) -E\left(u^{n}\right) \le 0$, it is sufficient to prove that \[
    \tau\sqrt{r-1}\left\|\left|\boldsymbol{c}^{n}_{r-1}\right|\right\| \le  \left\| u^{n+1} -u^{n}\right\|.
\]
By the triangle inequality $\left\| u^{n+1}-u^{n} \right\| \ge \left\| w_{1}^{n}(\tau)-u^{n} \right\| -\left\| u^{n+1}-w_{1}^{n}(\tau) \right\|$, 
where $w_{1}^{n}(s)$ represents the solution of ETDRK1,
we can infer that if the following inequality
\begin{equation}\label{eq:aim_0}
    \tau\sqrt{r-1} \left\|\left|\boldsymbol{c}^{n}_{r-1}\right|\right\| + \left\| u^{n+1}- w^{n}_{1}(\tau) \right\| \le \left\|w^{n}_{1}(\tau)-u^{n} \right\|
\end{equation}
holds, then $E\left(u^{n+1}\right) -E\left(u^{n}\right) \le 0$. 
%In other words, \cref{eq:aim_0} is a sufficient condition to show that $E\left( u^{n+1}\right) -E\left( u^{n}\right) \le 0$.
From the explicit formula of ETDRK$r$ \cref{eq:explicit_fomula_tau}, we have
\begin{equation}\label{eq:ur-u1}
    u^{n+1} - w_{1}^{n}(\tau) = w_{r}^{n}(\tau) - w_{1}^{n}(\tau) = \tau\sum_{k=1}^{r-1}k!{\phi_{k+1}}(\tau\mathcal{L} _{\kappa})c_{r-1,k}^n.
\end{equation}
Using the Cauchy inequality $\left\|\sum_{i=1}^{m}v_i\right\|^{2}\le m \sum_{i=1}^{m} \left\|v_i\right\|^{2}$ for any integer $m\ge 1$, we have
\begin{align}\label{eq:norm ur-u1}
\notag    \left\|u^{n+1}-w_{1}^{n}(\tau)\right\|
    & = \tau \left\|\sum_{k=1}^{r-1}k!{\phi_{k+1}}(\tau\mathcal{L} _{\kappa})c_{r-1,k}^n\right\| 
    \le  \tau  \left[(r-1)\sum_{k=1}^{r-1}  \left\|k!{\phi_{k+1}}(\tau\mathcal{L} _{\kappa})c_{r-1,k}^n\right\|^{2}  \right] ^{\frac{1}{2}} \\
    &\le  \tau \sqrt{r-1}  \left(\sum_{k=1}^{r-1}  \left\|\frac{c_{r-1,k}^n}{k+1}\right\|^{2}  \right) ^{\frac{1}{2}} \le \frac{1}{2}   \tau \sqrt{r-1}    \left(\sum_{k=1}^{r-1}  \left\|{c_{r-1,k}^n}\right\|^{2}  \right) ^{\frac{1}{2}} \\
\notag    & = \frac{1}{2}   \tau \sqrt{r-1}    \left\|\left|\boldsymbol{c}^{n}_{r-1}\right|\right\|,
\end{align} 
where we have used Lemma \ref{lemma:psi_k bound} in the second inequality. 
Based on inequalities \cref{eq:aim_0,eq:norm ur-u1}, 
we can infer that if the following inequality
\begin{equation}\label{eq:aim_1}
\boxed{
    \frac{3}{2} \tau\sqrt{r-1} \left\|\left|\boldsymbol{c}^{n}_{r-1}\right|\right\| \le \left\| w_{1}^{n}(\tau)-u^{n} \right\|}
\end{equation}
holds, then $E\left(u^{n+1}\right) -E\left(u^{n}\right) \le 0$. 
Note that the right-hand side of \cref{eq:aim_1} is independent of $r$.

{\em Part II}. We aim to employ mathematical induction to demonstrate that, for any $r \ge 1$ there exists a specific constant, denoted as $\tau_{\max,r}$, such that the inequality \cref{eq:aim_1} holds true for all $0<\tau \le \tau_{\max,r}$.

Firstly, let us verify the case of ETDRK1.
When $r=1$, the left-hand side of \cref{eq:aim_1} is equal to 0, so \cref{eq:aim_1} holds for any $\tau>0$.
Therefore, $\tau_{\max, 1} = + \infty$.

Suppose that for fixed $r\geq 1$, there exists certain constant $\tau_{\max,r}>0$, such that \cref{eq:aim_1} holds for any $\tau \le  \tau_{\max,r}$.
Then we consider the case of $r+1$ based on the inductive hypothesis, i.e. we want to show that, there exists certain positive constant $\tau_{\max,r+1} \le  \tau_{\max,r}$, such that
\begin{equation}\label{eq:aim r}
    \frac{3}{2} \tau\sqrt{r} \left\|\left|\boldsymbol{c}^{n}_{r}\right|\right\| \le \left\| w_{1}^{n}(\tau)-u^{n}\right\|
\end{equation}
holds for any $\tau \le  \tau_{\max,r+1}$.
According to the interpolation \cref{eq:vandermonde V_r}, we have
\[
\boldsymbol{c}^{n}_{r}=V_r^{-1}\boldsymbol{d}^{n}_{r}.
\]
Note that $V_{r}$ is a Vandermonde matrix with different nodes and $V_{r}^{-1}$ exists. 
According to the theory of singular value decomposition (SVD), for any $A \in \mathbb{R}^{d \times d}$ and $\boldsymbol{v} \in \mathbb{R}^{d}$, for the 2-norm of vector $\left|\cdot\right| $, we have
\begin{equation}\label{eq:SVD}
  \left|A\boldsymbol{v}\right|
  \le \sigma_{\max}(A) \left|\boldsymbol{v}\right|,
\end{equation}
where $\sigma_{\max}\left(A\right) $ represents the maximum singular value of  $A$. 
Therefore, we have
\begin{equation}\label{eq:br sigma}
  \left|\boldsymbol{c}^{n}_{r}\right| = \left|V_r^{-1}\boldsymbol{d}^{n}_{r}\right|
  \le \sigma_{\max} \left(V_r^{-1}\right)  \left|\boldsymbol{d}^{n}_{r}\right|
  = \frac{1}{\sigma_{\min} \left(V_r\right)}  \left|\boldsymbol{d}^{n}_{r}\right|,
\end{equation}
where $\sigma_{\min} $ is the minimum singular value of a matrix. 
By the Lipschitz condition of $f$, we have $\left|\mathcal{N}(u)-\mathcal{N}(v)\right| \le 2\kappa\left|u-v\right|$ for any given $u$ and $v$.
Therefore,
\begin{equation}\label{eq:sum norm br} 
\begin{aligned}
    \left\|\left|\boldsymbol{d}^{n}_{r}\right|\right\|
    &=\left\|\left(\sum_{k=1}^{r}\left( \mathcal{N}\left[w^{n}_{r}(a_{r,k}\tau)\right] -\mathcal{N}(u^{n})\right)^{2} \right)^{\frac{1}{2}} \right\|\\
    &=\left(\sum_{k=1}^{r}\left\|\mathcal{N}\left[w^{n}_{r}(a_{r,k}\tau)\right]-\mathcal{N}(u^{n})  \right\|^{2}\right)^{\frac{1}{2}} \\
    &\le 2\kappa\left(\sum_{k=1}^{r}\left\|w^{n}_{r}(a_{r,k}\tau)-u^{n}\right\|^{2}\right)^{\frac{1}{2}}. 
\end{aligned}
\end{equation}
Thus, the left-hand side of \cref{eq:aim r} satisfies the following inequality
\begin{equation}\label{eq:induction_1}
    \frac{3}{2}\tau\sqrt{r}\left\|\left|\boldsymbol{c}^{n}_{r}\right|\right\| \le 
    \frac{3\sqrt{r}\kappa\tau}{\sigma_{\min}\left(V_r\right)}   \left(\sum_{k=1}^{r}\left\|w^{n}_{r}(a_{r,k}\tau)-u^{n}\right\|^{2}\right)^{\frac{1}{2}}. 
\end{equation}
According to the explicit formula of ETDRK$r$ \cref{eq:explicit_fomula}, for $k=1,2, \ldots ,r$, we have
\begin{equation}\label{eq:u k/r - u n}
\begin{aligned}
    w^{n}_{r}(a_{r,k}\tau)-u^{n} 
    &= w^{n}_{1}(a_{r,k}\tau) - u^{n} + \tau \sum_{j=1}^{r-1} j! \left(a_{r,k}\right)^{j+1} \phi_{j+1}\left(a_{r,k}\tau\mathcal{L}_{\kappa}\right)c^{n}_{r-1,j}.
\end{aligned}
\end{equation}

Following the approach in the proof of Lemma \ref{lemma:psi_k bound}, for any real number $\lambda \in (0,1)$, we have
\begin{equation}\label{eq:noname_4}
\begin{aligned}
    \left\|w_{1}^{n}(\lambda\tau)-u^{n} \right\|
    &=\left\|\tau \lambda\phi_1(\lambda\tau\mathcal{L}_\kappa) \left(\mathcal{L}_\kappa u^{n}+ \mathcal{N}(u^{n})\right)  \right\| \\
    &=\left\|\tau h_1(\lambda,\tau\mathcal{L}_\kappa)    \left(\mathcal{L}_\kappa u^{n}+ \mathcal{N}(u^{n})\right)  \right\| \\
    &\le \left\|\tau h_1(1,\tau\mathcal{L}_\kappa)    \left(\mathcal{L}_\kappa u^{n}+ \mathcal{N}(u^{n})\right)  \right\| \\
    &= \left\|\tau \phi_1(\tau\mathcal{L}_\kappa) \left(\mathcal{L}_\kappa u^{n}+ \mathcal{N}(u^{n})\right)  \right\| \\
    &= \left\|w_{1}^{n}(\tau)-u^{n} \right\|,
\end{aligned}
\end{equation}
where the function $h_{1}(\lambda,x)$ is defined as in \cref{eq:h_k}. Therefore, we have
\begin{equation}\label{eq:noname_5}
    \left\|w_{1}^{n}(a_{r,k}\tau)-u^{n}\right\| \le \left\|w_{1}^{n}(\tau)-u^{n}\right\|.
\end{equation}

Combining \cref{eq:u k/r - u n,eq:noname_5}, we have
\begin{align}\label{eq:noname_3}
    \notag &  \sum_{k=1}^{r}\left\|w_{r}^{n}(a_{r,k}\tau)-u^{n}\right\|^{2} \\
    \notag &\le  \sum_{k=1}^{r} \left[\left\|w_{1}^{n}(a_{r,k}\tau)-u^{n}\right\|+ \tau\left\| \sum_{j=1}^{r-1}j! \left({a_{r,k}}\right)^{j+1} \phi_{j+1}\left(a_{r,k}\tau\mathcal{L}_{\kappa}\right)c^{n}_{r-1,j}\right\|\right]^{2} \\
    &\le  \sum_{k=1}^{r} \left[\left\|w_{1}^{n}(\tau)-u^{n}\right\|+\tau \left\| \sum_{j=1}^{r-1}j! \left({a_{r,k}}\right)^{j+1} \phi_{j+1}\left(a_{r,k}\tau\mathcal{L}_{\kappa}\right)c^{n}_{r-1,j}\right\|\right]^{2} \\
    \notag &\le \sum_{k=1}^{r}  \left[\left\|w_{1}^{n}(\tau)-u^{n} \right\|+\tau
    \left((r-1)\sum_{j=1}^{r-1}\left\| j! \left({a_{r,k}}\right)^{j+1} \phi_{j+1}\left(a_{r,k}\tau\mathcal{L}_{\kappa}\right)c^{n}_{r-1,j}\right\|^2 \right)^{\frac{1}{2}}
    \right]^{2} \\
    \notag &\le \sum_{k=1}^{r}  \left[\left\|w_{1}^{n}(\tau)-u^{n} \right\|+\tau
    \left((r-1)\sum_{j=1}^{r-1}\left\| \frac{c
    ^{n}_{r-1,j}}{j+1} \right\|^2 \right)^{\frac{1}{2}}
    \right]^{2} \\
    \notag 
    &\le 
    r \left(\left\|w_{1}^{n}(\tau)-u^{n} \right\|+ \frac{1}{2}\tau \sqrt{r-1}
     \left\|\left|\boldsymbol{c}^{n}_{r-1}\right|\right\|
    \right)^{2} .
\end{align}
Thus, combining \cref{eq:induction_1} and \cref{eq:noname_3}, we have
\begin{equation}\label{eq:induction_2}
\begin{aligned}
    &\frac{3}{2}\tau\sqrt{r} \left\|\left|\boldsymbol{c}^{n}_{r}\right|\right\|  \le \frac{3r\kappa\tau}{\sigma_{\min}\left(V_r\right)}  \left(\left\|w_{1}^{n}(\tau)-u^{n} \right\|+ \frac{1}{2}\tau \sqrt{r-1}
     \left\|\left|\boldsymbol{c}^{n}_{r-1}\right|\right\| \right).  
\end{aligned}
\end{equation}
From inductive hypothesis \cref{eq:aim_1}, we know
\begin{equation}\label{eq:induc_hypothesis}
    \left\|\left|\boldsymbol{c}^{n}_{r-1}\right|\right\|  \le \frac{2}{3\tau \sqrt{r-1}} \left\|w_{1}^{n}(\tau)-u^{n}\right\|.
\end{equation}
Substituting \cref{eq:induc_hypothesis} into \cref{eq:induction_2}, we have
\begin{equation}\label{eq:noname_7}
    \begin{aligned}
        \frac{3}{2}\tau\sqrt{r} \left\|\left|\boldsymbol{c}^{n}_{r}\right|\right\|  
        \le  \frac{4r\kappa\tau}{\sigma_{\min}\left(V_r\right)}\left\|w_{1}^{n}(\tau)-u^{n} \right\|.    
    \end{aligned}
\end{equation}
To  achieve our aim \cref{eq:aim r}, it is sufficient to satisfy the following inequality:
\begin{equation}\label{eq:tempb1}
    \frac{4r\kappa\tau}{\sigma_{\min}\left(V_r\right)} \le 1.
\end{equation}
Therefore, then \cref{eq:aim r} holds for any
\[
    \tau \le  \tau_{\max,r+1} \coloneqq \min \left\{\tau_{\max,r}, \frac{\sigma_{\min}\left(V_r\right)}{4r\kappa}  \right\}.
\]
Since $\tau_{\max,1}=+\infty$,
we have 
\begin{equation}\label{eq:tau_max_r}
    \tau_{\max,r} = \frac{1}{4\kappa} \min \left\{{\sigma_{\min}\left(V_1\right)},\frac{\sigma_{\min}\left(V_2\right)}{2}, \ldots ,\frac{\sigma_{\min}\left(V_{r-1}\right)}{r-1}\right\}.
\end{equation}
Here, $V_{r}$ is the Vandermonde matrix of interpolation nodes, and one can use different nodes for different $r$. 
For example, if we use uniform nodes for any $r$, it is easy to know $\sigma_{\min}\left(V_{r-1} \right) \le \sigma_{\min}\left(V_r\right)$. 
Then the time-step restriction of uniform nodes is 
\[
   \tau \le  \tau_{\max,r} = \frac{\sigma_{\min}\left(V_{r-1}\right)}{4\kappa(r-1)}.
\]
Finally, by induction, we finish the proof. 
\end{proof}

\begin{remark}
    In this proof, we have used many coarse inequalities and the final coefficients might be further improved. 
    Readers interested in refining these coefficients are invited to explore further.
    However, the determining factor influencing the magnitude of $\tau_{\max,r}$ is the minimum singular value of $V_{r-1}$, and the impact of other coefficients is much smaller.
    From \cref{eq:tempb1}, we know that the maximum allowable time step, $\tau_{\max,r}$, is approximately proportional to the minimum singular value, $\sigma_{\min}(V_{r-1})$. 
    The Vandermonde matrix $V_{r}$, derived from the interpolation, depends on the interpolation nodes.
    We present the minimum singular values of Vandermonde matrices with uniform and Chebyshev nodes, in Table \ref{table:sigma_min_two_nodes}.
    It is easy to see that $\sigma_{\min}(V_{r-1})$ for both node types diminish exponentially with increasing $r$, significantly constraining the time-step size as $r$ increases.
    Since the magnitudes of both nodes are close, we use uniform nodes in the following paper.
    We present $\tau_{\max,r}$ for the Allen--Cahn equation with the Ginzburg--Landau function $f_{GL}=u-u^{3}$, $\kappa=2$ and uniform nodes in Table \ref{table:taur poly}.
    As showed, $\tau_{\max,r}$ decreases almost exponentially. However, our numerical experiments show that, even with larger time steps, the high-order ETDRK schemes still preserve the original energy dissipation law. This observation suggests that our theoretical analysis of $\tau_{\max}$ may not be sharp, and motivates us to relax the step-size constraint in future work.

    Fu et al. \cite{FU-Yang2022ETDRK2Energy,FU-SHEN-YANG2024HigherOrder_ED} proved that the ETDRK2 scheme using nodes $\{a_{1,0}=0,a_{1,1}=1\}$ unconditionally decreases the original energy. 
    In addition, they found special types of ETDRK3 schemes which unconditionally decrease the original energy in \cite{FU-SHEN-YANG2024HigherOrder_ED}.
    Due to the difficulty in satisfying the positive definiteness conditions in their theorem, an ETDRK4 scheme which unconditionally decreases the original energy has not yet been discovered.
    However, our new analysis of original energy dissipation is applicable for all higher-order ETDRK schemes with arbitrary interpolation nodes within $[0,1]$. 
\end{remark}

\begin{table}[htbp]
    \centering
    \footnotesize
    \caption{The minimum singular values of the Vandermonde matrix using uniform and Chebyshev nodes.}
    \label{table:sigma_min_two_nodes}
    \begin{tabular}{c|cccccc}
    \hline
    $r$ & 1 & 2 & 3 & 4 & 5 \\ \hline 
    $\sigma_{\min}(V_{r}) ~ \text{(Uniform)} $ & 1.000e+00 & 1.654e-01 & 2.745e-02 & 4.408e-03 & 6.807e-04 \\ \hline
        $\sigma_{\min}(V_{r}) ~ \text{ (Chebyshev)} $ & 1.000e+00 & 1.654e-01 & 3.395e-02 & 6.823e-03 & 1.338e-03 \\ \hline 
    \hline 
    $r$ & 6 & 7 & 8 & 9 & 10 \\ \hline 
    $\sigma_{\min}(V_{r}) ~ \text{(Uniform)}$ & 1.017e-04 & 1.481e-05 & 2.113e-06 & 2.971e-07 & 4.125e-08 \\ \hline 
    $\sigma_{\min}(V_{r}) ~ \text{ (Chebyshev)} $ & 2.575e-04 & 4.884e-05 & 9.157e-06 & 1.701e-06 & 3.136e-07 \\ \hline
    \end{tabular}
\end{table}

% \begin{table}[htbp]
%     \centering
%     \footnotesize
%     \caption{The minimum singular values of the Vandermonde matrix using uniform nodes.}
%     \label{table:sigma_min_uniform_nodes}
%     \begin{tabular}{c|cccccc}
%     \hline
%     $r$ & 1 & 2 & 3 & 4 & 5 \\ \hline 
%     $\sigma_{\min}(V_{r})$ & 1.000e+00 & 1.654e-01 & 2.745e-02 & 4.408e-03 & 6.807e-04 \\ \hline \hline 
%     $r$ & 6 & 7 & 8 & 9 & 10 \\ \hline 
%     $\sigma_{\min}(V_{r})$ & 1.017e-04 & 1.481e-05 & 2.113e-06 & 2.971e-07 & 4.125e-08 \\ \hline 
%     \end{tabular}
% \end{table}

% \begin{table}[htbp]
%     \centering
%     \footnotesize
%     \caption{The minimum singular values of the Vandermonde matrix using Chebyshev nodes.}\label{table:sigma_min_chebyshev_nodes}
%     \begin{tabular}{c|cccccc}
%     \hline
%     $r$ & 1 & 2 & 3 & 4 & 5 \\ \hline 
%     $\sigma_{\min}(V_{r})$ & 1.000e+00 & 1.654e-01 & 3.395e-02 & 6.823e-03 & 1.338e-03 \\ \hline \hline
%     $r$ & 6 & 7 & 8 & 9 & 10 \\ \hline 
%     $\sigma_{\min}(V_{r})$ & 2.575e-04 & 4.884e-05 & 9.157e-06 & 1.701e-06 & 3.136e-07 \\ \hline 
%     \end{tabular}
% \end{table}

\begin{table}[htbp] 
    \centering
    \footnotesize
    \caption{The values of $\tau_{\max,r}$ for the Allen--Cahn equation with $f_{\text{GL}}(u)=u-u^{3}$, $\kappa=2$ and uniform interpolation nodes.}
    \label{table:taur poly}
    \begin{tabular}{c|ccccc}
    \hline
    $r$ & 1 & 2 & 3 & 4 & 5 \\ \hline 
    $\tau_{\max,r}$ &  $+\infty$ & 1.250e-01 & 1.034e-02 & 1.144e-03 & 1.378e-04 \\ \hline  \hline
    $r$ & 6 & 7 & 8 & 9 & 10 \\ \hline 
    $\tau_{\max,r}$ & 1.702e-05 & 2.118e-06 & 2.644e-07 & 3.302e-08 & 4.126e-09 \\ \hline 
    \end{tabular}
\end{table}

% \begin{figure}[htbp] 
%     \centering
%     \includegraphics[scale=0.8]{figs/fig1_kappatau_max.pdf}
%     \caption{The variation in maximum time-step size, $\kappa\tau_{\max,r}$, as a function of ${\sigma_{\min}(V_r)}/{r}$.}
%     \label{fig:111}
% \end{figure}

\section{ETDRK methods with rescaling technique}\label{section:rescaling technique}
In the proof of Theorem \ref{theorem:energy decay main theorem} and the proof of original energy dissipation law in the literature \cite{FU-SHEN-YANG2024HigherOrder_ED}, the assumption of a Lipschitz condition on the function $f$ is required. However, for commonly used functions such as the Ginzburg--Landau potential, $f_{\text{GL}}(u) = u - u^3$, and the Flory-Huggins potential, $f_{\text{FH}}(u) = {\theta} (\ln {(1-u)}-\ln {(1+u)})/2 + \theta_c u$ , the derivatives are unbounded across the real number domain. 
In other words, the Lipschitz condition assumption does not hold if the MBP is not preserved. 
It is necessary and important to preserve MBP for high-order ETDRK schemes.
For the Cahn--Hilliard equation, to obtain the Lipschitz condition, a common practice is to modify the potential by truncating it for $|x| > M$ (for a sufficiently large $M$) and smoothly integrating it with a quadratic function smoothly
connected to the inner part \cite{shen_yang2010AC_CH}.

% The Allen--Cahn equation is notable for its maximum bound principle (MBP), i.e. if the initial values are bounded by $\beta$ in absolute value, then the absolute value of the solution is also bounded by $\beta$ everywhere and for all time. 
% If the numerical solutions of ETDRK can preserve the MBP, then the Lipschitz condition assumption is automatically satisfied. 
% However, it has been established that only the lower-order ETDRK schemes, specifically the ETDRK1 and ETDRK2, can unconditionally preserve the MBP, as demonstrated for a class of parabolic equations \cite{DU2021ETD}. Numerical observations have further confirmed that third-order and higher-order ETDRK schemes do not preserve the MBP unconditionally. Designing arbitrarily high-order unconditionally MBP-preserving ETDRK methods remains a challenge.

This section introduces a rescaling technique aimed at unconditionally preserving the MBP for arbitrarily high-order ETDRK schemes. 
The key of this technique is to rescale the interpolation polynomial. Note that the unboundedness of the interpolation polynomial is thought to cause the exceeding of maximum bound. To preserve the MBP, we slightly rescale the approximation polynomial $P_{r-1}$ in a manner that does not affect the order of the interpolation error. Replacing the interpolation polynomial with this rescaled version, we obtain a modified $r$th-order ETDRK scheme, which we refer to as the ETDRK method with rescaling technique. In subsection \ref{section:Convergence analysis}, we will show that the ETDRK method equipped with rescaling technique preserves the MBP unconditionally at arbitrary high orders while maintaining the same accuracy as the original ETDRK method.
By integrating ETDRK schemes with this rescaling technique, we can eliminate the necessity for the Lipschitz condition assumption of $f$ in Theorem \ref{theorem:energy decay main theorem}.

\subsection{Rescaling technique}
It has been proven in \cite{DU-JU-LI-QIAO2021ETD} that if the stabilizing constant $\kappa \ge \max _{|\xi| \le \beta}\left|f^{\prime}(\xi)\right|$, the ETDRK1 and ETDRK2 preserve the MBP.
Now, we can remove the Lipschitz assumption \cref{eq:Lipschitz} on $f$, and instead, we only require that
\begin{equation}\label{eq:kappa_cond2}
    \kappa \ge \max _{|\xi| \le \beta}\left|f^{\prime}(\xi)\right|.
\end{equation}
Consider the ETDRK$r$ scheme \cref{eq:ETDRKr}, and let us define a scaling factor
\begin{equation}\label{eq:scaling factor} 
    \alpha_{r-1}^n(x) \coloneqq \min \left\{\frac{\kappa \beta}{\max\limits_{s \in [0,\tau]} \left|P^{n}_{r-1}(x,s)\right|}, 1\right\}, 
\end{equation}
where $\beta$ is the bound of the MBP. 
We want to replace $P^{n}_{r-1}$ with $\tilde{P}^{n}_{r-1} \coloneq \alpha_{r-1}^n(x)  {P}^{n}_{r-1}$.
Then we explain that the adjustment does not change the accuracy of schemes. In the following analysis, we consider a fixed $x \in \Omega$.
If $\max\limits_{s \in [0,\tau]} \left|P_{r-1}^{n}\right| \le \kappa \beta$, then $\alpha_{r-1}^n(x) = 1$, i.e., $P^{n}_{r-1}(x,s)$ doesn't need to be changed.
If $\max\limits_{s \in [0,\tau]} \left|P_{r-1}^{n}\right| > \kappa \beta$, then for any $s \in [0,\tau]$,
\begin{equation}
    \begin{aligned}
        \left|{P}_{r-1}^{n}(x,s) - \tilde{P}_{r-1}^{n}(x,s)\right|
        &= \frac{\max\limits_{s \in [0,\tau]} \left|P_{r-1}^{n}\right| -\kappa \beta }{\max\limits_{s\in[0,\tau]} \left|P^{n}_{r-1}\right|} \left|{P}_{r-1}^{n}(x,s)\right| \\
        &\le \max\limits_{s \in [0,\tau]} \left|P_{r-1}^{n}\right| -\kappa \beta.
    \end{aligned}
\end{equation}
If we denote
$
    s_{*}(x)  \coloneq \argmax _{s \in [0,\tau]} \left|P_{r-1}^{n}(x,s)\right|,
$
then for any $s \in [0,\tau]$,
\begin{equation}\label{eq:tempc1}
    (1-\alpha_{r-1}^n(x))\left|{P}_{r-1}^{n}(x,s)\right| \le \left| P_{r-1}^{n}(x,s_{*})\right| -\kappa \beta .
\end{equation}
Since $\kappa \ge \max _{|\xi| \le \beta}\left|f^{\prime}(\xi)\right|$, it is easy to check that 
%\begin{equation}\label{eq:Npri_positive}
$0 \le \mathcal{N}^{\prime}(\xi) \le 2 \kappa,~  \forall \xi \in [-\beta, \beta].$
%\end{equation}
Thus, for any $\xi \in [-\beta, \beta]$, we know that
\begin{equation}
    -\kappa \beta \le f(-\beta)-\kappa \beta=\mathcal{N}(-\beta) \le \mathcal{N}(\xi) \le \mathcal{N}(\beta)=f(\beta)+\kappa \beta \le \kappa \beta.
\end{equation}
 Therefore, $\mathcal{N}(u(t_n+s_{*})) \in [-\kappa\beta, \kappa\beta]$ because the exact solution $u$ preserves the MBP.
Referring to  \cref{eq:tempc1}, we obtain
\begin{equation}\label{eq:same order}
    \begin{aligned}
        \left|{P}_{r-1}^{n}(x,s) - \tilde{P}_{r-1}^{n}(x,s)\right|
        &\le \left| P_{r-1}^{n}(x,s_{*})| - |\mathcal{N}(u(x,t_n+s_{*}))\right|\\
        &\leq \left|P_{r-1}^{n}(x,s_{*}) - \mathcal{N}(u(x,t_n+s_{*}))\right| 
         = O(\tau^{r}),
    \end{aligned}
\end{equation}
for any $s \in [0,\tau]$.
% Here, the last equality arises from the error estimation of interpolation. 
% The error estimation in \cref{eq:same order} is crucial for ensuring that the rescaling technique does not lose accuracy because it keeps the same order as the truncation error in \cite{DU-JU-LI-QIAO2021ETD}. 
The last equality comes from interpolation error estimation. The error estimation in \cref{eq:same order} is crucial for ensuring that the rescaling technique does not lose accuracy, since it keeps the same order as the truncation error in \cite{DU-JU-LI-QIAO2021ETD}.

Replacing $P^{n}_{r-1}(x,s)$ with $\tilde{P}^{n}_{r-1}(x,s) = \alpha_{r-1}^n(x)  {P}^{n}_{r-1}(x,s)$ in \cref{eq:ETDRKr}, we derive the ETDRK$r$ scheme equipped with the rescaling technique, i.e.
for $n \geq 0$ and given $u^n$, find $u^{n+1}=w_r^{n}(\tau)$ by solving:
\begin{equation}\label{eq:ETDRKr rescaling}
    \left\{
    \begin{aligned}
       & \partial_s w_{r}^{n}  =\mathcal{L}_{\kappa} w_{r}^{n}+\tilde{P}^{n}_{r-1}(s), & &  {x} \in \Omega , ~ s \in(0, \tau], \\
        & w_r^{n}(0, {x})        =u^n({x}),                                              & & {x} \in \bar{\Omega},
    \end{aligned}
    \right.
\end{equation}
equipped with homogeneous Neumann boundary condition, where $u^n$ is the numerical solution approximating $u\left(t_n\right)$, and $u^0$ is given in \cref{eq:AC}. 

We now introduce a useful lemma concerning the contraction semigroup, which plays a crucial role in preserving the MBP. For the Laplace operator $\Delta$ with homogeneous Neumann boundary condition, 
and the maximum norm   
as
$
    \left\|u\right\|_{\infty} \coloneq \max_{x \in \bar{\Omega}} \left|u(x)\right| $ for any  $ u \in C(\bar{\Omega})
$,
we have the following lemma.

% and define the maximum norm $\left\|\cdot \right\|_{\infty}$ as
% \[
% \left\|u\right\|_{\infty} = \max_{x\in\bar{\Omega}}|u(x)| , \quad \forall u \in C(\bar{\Omega}).
% \]

\begin{lemma}[\cite{DU-JU-LI-QIAO2021ETD}]\label{lem:contraction semi-group} 
 The Laplace operator $\Delta$ with the periodic or homogeneous Neumann boundary condition generates a contraction semigroup $\left\{S_{\Delta}(t)=e^{t \Delta}\right\}_{t \geq 0}$ with respect to the maximum norm on $C(\bar{\Omega})$. Moreover, for a positive real number $\gamma$, there holds
$$
\left\|e^{t(\Delta-\gamma)} u\right\|_{\infty} \leq e^{-\gamma t}\left\|u\right\|_{\infty}, \quad \forall t \geq 0, u \in C(\bar{\Omega}).
$$
\end{lemma}
The proof details of this lemma can be found in \cite{DU-JU-LI-QIAO2021ETD}. The following theorem proves that the ETDRK method with the rescaling technique unconditionally preserves the MBP.

\begin{theorem}\label{theorem:MBP}
Suppose that \cref{eq:f(beta)<0,eq:kappa_cond2} hold. If the initial value $u^0$ of the Allen--Cahn equation \cref{eq:AC} satisfies $\left\|u^{0}\right\|_{\infty} \le \beta$, then the ETDRK$r$ method equipped with the rescaling technique \cref{eq:ETDRKr rescaling} preserves the MBP unconditionally, i.e. for any time-step size $\tau>0$,
    the solution of ETDRK$r$ schemes with the rescaling technique \cref{eq:ETDRKr rescaling} satisfies $ \left\|u^n\right\|_{\infty} \le \beta$. 
\end{theorem}
\begin{proof}
    To establish the MBP, it suffices to prove $\left\|u^{n+1}\right\|_{\infty} \le \beta$ if $\left\|u^n\right\|_{\infty} \le \beta$. 
    For a fixed $x\in\Omega$, if $\max\limits_{s \in [0,\tau]} \left|P_{r-1}(x,s)\right| > \kappa\beta$, then we have 
    \begin{equation} 
        \begin{aligned}
            \left|\tilde{P}^{n}_{r-1}(x,s) \right| 
            =  \alpha_{r-1}^n(x) | {P}_{r-1}(x,s)|
            =  \frac{\kappa \beta}{\max\limits_{s \in [0,\tau]} \left|P_{r-1}\right|} | {P}_{r-1}|
            \le \kappa \beta,
        \end{aligned}
    \end{equation}
    for any $s \in [0,\tau]$. 
    Conversely, if $\max\limits_{s \in [0,\tau]} \left|P_{r-1}(x,s)\right| \le \kappa\beta$, then $\alpha_{r-1}^n(x)=1$ and
    $\tilde{P}^{n}_{r-1}= {P}^{n}_{r-1}$. Regardless, we establish that
    \begin{equation} \label{eq:<=kappa beta}
        \begin{aligned}
            \left|\tilde{P}^{n}_{r-1}(x,s) \right| \le \kappa \beta, \quad \forall (x,s) \in \Omega \times [0,\tau].
        \end{aligned}
    \end{equation}
By Duhamel's principle, we know that the explicit solution of \cref{eq:ETDRKr rescaling} is
\begin{equation}
w^{n}_{r}(s)=\mathrm{e}^{s \mathcal{L}_\kappa} u^n+  \int_0^s \mathrm{e}^{(s-\sigma) \mathcal{L}_\kappa} \tilde{P}^{n}_{r-1}(\sigma) \mathrm{~d}  \sigma .
\end{equation}
Using Lemma \ref{lem:contraction semi-group}, along with \cref{eq:<=kappa beta} and  $\left\|u^n\right\|_{\infty} \le \beta$, for any $s \in [0,\tau]$, we have
    \begin{align}\label{aeq_t1}
    \left\|w^{n}_{r}(s)\right\|_{\infty} 
    \notag & \le \mathrm{e}^{-\kappa s}\left\|u^n\right\|_{\infty}
    +\int_0^s\mathrm{e}^{-\kappa(s-\sigma)} \left\| \tilde{P}^{n}_{r-1}(\sigma) \right\|_{\infty} \mathrm{~d} \sigma \\
    & \le \beta \mathrm{e}^{-\kappa s} +\int_0^s  \mathrm{e}^{-\kappa(s-\sigma)} \kappa \beta \mathrm{~d}  s \\
    \notag & = \beta \mathrm{e}^{-\kappa s}+\kappa \beta \frac{1-\mathrm{e}^{-\kappa s}}{\kappa}
    =\beta.
    \end{align}
Since $u^{n+1}=w^{n}_{r}(\tau)$, \cref{aeq_t1} shows the ETDRK methods equipped with the rescaling technique are unconditionally MBP-preserving.
\end{proof}
\begin{remark}
Since $\kappa \ge \max _{|\xi| \le \beta}\left|f^{\prime}(\xi)\right|$, the function $\mathcal{N}(u(t_{n}+s))$ lies within the interval $[-\kappa \beta, \kappa \beta]$. However, the interpolation polynomial $P_{r-1}(s)$ does not necessarily maintain its values within $[-\kappa \beta, \kappa \beta]$. This is the primary reason the ETDRK numerical method fails to preserve the MBP. 
There are only two instances where $|P_{r-1}|$ is guaranteed to be bounded by $\kappa\beta$. In the first instance, for $r=1$, the polynomial $P_0 = \mathcal{N}(u(t_{n}))$, is constant and remains within the interval $[-\kappa\beta, \kappa\beta]$. In the second instance, for $r=2$ with nodes ${a_{1,0}=0, a_{1,1}=1}$, the polynomial $P_1(s) = \mathcal{N}(u(t_{n})) + s \left( {\mathcal{N}(u(t_{n+1})) - \mathcal{N}(u(t_n))} \right)/{\tau}$ also stays within this interval. The two instances correspond to ETDRK1 and ETDRK2 schemes, respectively. Both schemes have been proven to preserve the MBP unconditionally in \cite{DU-JU-LI-QIAO2021ETD} and the original energy dissipation law unconditionally in \cite{FU-Yang2022ETDRK2Energy}. 
Apart from these two specific cases, there is no guarantee that the interpolation polynomials will remain within $\left[-\kappa\beta, \kappa\beta\right]$. Therefore, to preserve the MBP, it is necessary to implement some strategies to ensure that the interpolation polynomials are adjusted to stay within the range.
\end{remark}
\begin{remark}
    Another common MBP-preserving strategy for high-order schemes is the cut-off technique, see for example \cite{LI-YANG-ZHOU2020cutoff,YANG-YUAN-ZHOU2022cutoff}. As depicted in Figure \ref{fig:compare_cut_scale}, applying the cut-off technique to the interpolation polynomial $P(s)$ results in a bounded function $\tilde{P}(s)$, which is no longer a polynomial. Consequently, the Integral of $\mathrm{e}^{(\tau-s) \mathcal{L}_\kappa} \tilde{P}(s)$ cannot be expressed with an exact formula. To address this, we introduce the rescaling technique, ensuring that $\tilde{P}(s)$ remains a polynomial. The current cut-off technique is typically applied to the unknown function $u(s)$ rather than to the interpolation function $P(s)$, making the analysis of the energy dissipation law difficult. For more detailed explorations of the cut-off technique, one can refer to \cite{LI-YANG-ZHOU2020cutoff,YANG-YUAN-ZHOU2022cutoff}.
\end{remark}
\begin{figure}[!ht]
    \centerline{
    \hspace{-0.2cm}
    \includegraphics[width=0.4\textwidth]{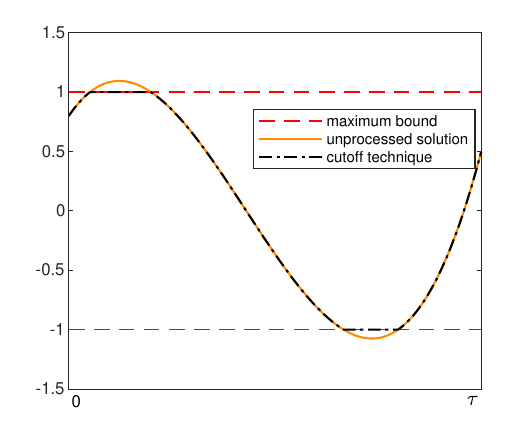}\hspace{-0.2cm}
    \includegraphics[width=0.4\textwidth]{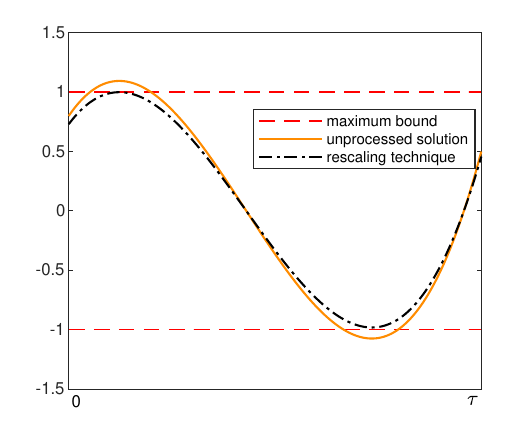}}
    \vspace{-0.2cm}
    \caption{Comparison of the cut-off (left) and rescaling (right) techniques.}
    \label{fig:compare_cut_scale}
\end{figure}
\begin{remark}
    Since the ETDRK method is a single-step method, employing the rescaling technique ensures preserving the MBP for any time mesh. However, when dealing with multi-step methods especially with non-uniform time mesh, preserving MBP unconditionally remains a significant and challenging issue. Some progress have been made in addressing this issue, particularly concerning the backward differentiation formula (BDF)  \cite{chen2019BDF,LIAO-TANG-ZHOU2020BDF,liao2022BDF2,akrivis2024variable}.
\end{remark}

\subsection{Energy dissipation of ETDRK methods with rescaling technique}
For ETDRK schemes employing the rescaling technique, 
we show that it retain the property of original energy decay as stated in Theorem \ref{theorem:energy decay main theorem}. The conclusions remain unchanged, with only small adjustments needed in the proof. In other words, using the rescaling technique, we can eliminate the requirement of the Lipschitz condition for $f$ in Theorem \ref{theorem:energy decay main theorem}.
\begin{theorem}\label{theorem:energy decay rescaling}
Suppose that \cref{eq:f(beta)<0,eq:kappa_cond2} hold.
Then the ETDRK$r$ scheme with the rescaling technique preserves the original energy dissipation law within a certain time-step size restriction.
More precisely, for any integer $r\ge 1$, there exists a positive constant $\tau_{\max,r}$ independent of $\varepsilon$, such that
    the solution $\left\{u^{n}\right\}_{n\ge 0} $ to the ETDRK$r$ scheme with the rescaling technique satisfies
        \[
            E\left(u^{n+1}\right) - E\left(u^{n}\right) \coloneqq E\left(w^{n}_{r}(\tau)\right) - E\left(w^{n}_{r}(0)\right) \le 0 , 
        \]
     for all $\tau\le \tau_{\max,r}$. The time-step size restriction is  
\[
    \tau_{\max,1} = + \infty, \quad 
    \tau_{\max,r} = \frac{1}{10\kappa} \min \left\{\sigma_{\min}\left(V_1\right),\frac{\sigma_{\min}\left(V_2\right)}{2}, \ldots ,\frac{\sigma_{\min}\left(V_{r-1}\right)}{r-1}\right\}, \quad r\ge 2, 
\]
     where $\sigma_{\min}(A)$ represents the minimum singular value of a matrix $A$, and $V_{k}$ is the $k \times k$ Vandermonde matrix of interpolation nodes defined in \cref{eq:vandermonde V_r}. 
%     \[
%     \tau_{\max,1} = + \infty, \quad \tau_{\max,r} = \frac{1}{\kappa} \min \left\{\mu_{1}^{2},\mu_{2}^{2}, \ldots ,\mu_{r-1}^{2}\right\}, \quad r\ge 2,
%     \]
% where $\mu_{k}$ is the unique root of the equation $\displaystyle x^{2}\left(2x^{2}+2 \sqrt{2}{x}+(4+\kappa)\right) = \frac{4 \sigma^{2}_{\min}(V_r)}{9r^{2}}$.
% Here, $\sigma_{\min}(A)$ represents the minimum singular value of a matrix $A$, and $V_{k}$ is the $k \times k$ Vandermonde matrix of interpolation defined in \cref{eq:vandermonde V_r}.
\end{theorem}
\begin{proof}
The proof is quite similar to the proof of Theorem \ref{theorem:energy decay main theorem}. 
Following with the analysis framework in \cref{eq:E_2-E_1 process}, we have
\begin{align}\label{eq:ED_rescaling}
   \notag & E\left( u^{n+1}\right) -E\left( u^{n}\right) \\
   \notag &=\left(\int _{0}^{\tau }( -\mathcal{L}_{\kappa } w^{n}_{r} -\mathcal{N}( w^{n}_{r}))  \partial_s w_{r}^{n} \mathrm{~d} s,1\right) \\
 \notag   &\le \left( - \int _{0}^{\tau}\left(\partial_s  w _{r}^{n}\right)^{2} \mathrm{~d} s + \int _{0}^{\tau}\left(\alpha_{r-1}^{n}P_{r-1}^{n} (s) -\mathcal{N}(u^{n}) \right)\partial_s  w _{r}^{n} \mathrm{~d} s ,1\right)\\
 \notag   &\le \left( -\frac{1}{2} \int _{0}^{\tau}\left(\partial_s  w _{r}^{n}\right)^{2} \mathrm{~d} s +\frac{1}{2}\int _{0}^{\tau}\left(\alpha_{r-1}^{n}P_{r-1}^{n} (s) -\mathcal{N}(u^{n}) \right)^{2} \mathrm{~d} s ,1\right)\\
    &\le - \frac{1}{2\tau} \left\|u^{n+1}-u^{n}\right\|^{2} + \frac{1}{2} \left(\int_{0}^{\tau} \left( \alpha_{r-1}^{n}P_{r-1}^{n} (s)-\mathcal{N}( u^{n})\right)^{2} \mathrm{~d} s,1\right) \\
  \notag  & =  \frac{1}{2}\left( \int_{0}^{\tau} \left[\left(P_{r-1}^{n} -\mathcal{N}( u^{n})\right) - (1-\alpha_{r-1}^{n})P_{r-1}^{n}\right] ^{2}     \mathrm{~d} s, 1\right)  - \frac{1}{2\tau}\left\|u^{n+1}-u^{n}\right\|^{2}\\
 \notag   & \le  \left( \int_{0}^{\tau} \left[P_{r-1}^{n} -\mathcal{N}( u^{n})\right]^{2} + \left[(1-\alpha_{r-1}^{n})P_{r-1}^{n}\right]^{2}     \mathrm{~d} s, 1\right)  - \frac{1}{2\tau} \left\|u^{n+1}-u^{n}\right\|^{2}\\
  \notag  & \le \tau(r-1) \left\|\left|\boldsymbol{c}^{n}_{r-1}\right| \right\|  + \left( \int_{0}^{\tau}\left[(1-\alpha_{r-1}^{n})P_{r-1}^{n}\right]^{2}     \mathrm{~d} s, 1\right)  - \frac{1}{2\tau} \left\|u^{n+1}-u^{n}\right\|^{2}.
\end{align}
Since $\mathcal{N}(u^n) \in [-\kappa\beta,\kappa\beta]$, and similar with \cref{eq:same order}, for any $(x,s) \in \Omega \times [0,\tau]$, we have
\begin{equation}\label{ineq:aP}
(1-\alpha_{r-1}^n(x))\left|{P}_{r-1}^{n}(x,s)\right| 
\le \left| P_{r}^{n}(x,s_{*}(x)) - \mathcal{N}(u^n(x))\right|,
\end{equation}
where $s_{*}(x) = \argmax _{s \in [0,\tau]} \left|P_{r}^{n}(x,s)\right|$.
Therefore, we obtain
\begin{align} \label{eq:alpha_bound}
   \notag \left( \int_{0}^{\tau} \left[(1-\alpha_{r-1}^{n})P_{r-1}^{n}\right]^{2}     \mathrm{~d} s, 1\right) &\le 
    \left( \int_{0}^{\tau} \left( P_{r-1}^{n}(x,s_{*}(x)) - \mathcal{N}(u^n(x))\right)^{2}     \mathrm{~d} s, 1\right) \\
    &  = \tau  \left(\left( \sum _{k=1}^{r-1} \left(\frac{s_*}{\tau}\right)^{k} c_{r-1,k}^{n} \right)^{2},1\right) \\
   \notag &\le \tau(r-1) \left\|\left|\boldsymbol{c}^{n}_{r-1}\right| \right\|^2.
\end{align}
Substituting \cref{eq:alpha_bound} into \cref{eq:ED_rescaling}, we obtain
\begin{equation}\label{eq:ED_rescaling_2}
    \begin{aligned}
     E\left( u^{n+1}\right) -E\left( u^{n}\right) 
     \le\frac{1}{2\tau} \left[4\tau^2(r-1) \left\|\left|\boldsymbol{c}^{n}_{r-1}\right| \right\|^2  -  \left\|u^{n+1}-u^{n}\right\|^{2}\right] .
\end{aligned}
\end{equation}
Compared with \cref{eq:Aim}, only the coefficient of term $\left\|\left|\boldsymbol{c}^{n}_{r-1}\right| \right\|$ is multiplied by 4.
Similarly, we can infer that if the following inequality
\begin{equation}\label{eq:aim_0_rescale}
  2\tau\sqrt{r-1} \left\|\left|{\boldsymbol{c}}^{n}_{r-1}\right|\right\| + \left\| u^{n+1}- w^{n}_{1}(\tau) \right\| \le \left\|w^{n}_{1}(\tau)-u^{n} \right\|
\end{equation}
holds, then $E\left(u^{n+1}\right) -E\left(u^{n}\right) \le 0$.

From the explicit formula of ETDRK$r$ with the rescaling technique, we have
\begin{equation}
    u^{n+1} - w_{1}^{n}(\tau)  = \tau\phi_1(\tau\mathcal{L} _{\kappa}) 
        \left[ (\alpha_{r-1}^{n}-1)\mathcal{N}(u^{n}) \right]
          + \tau \sum_{k=1}^{r-1}k!{\phi_{k+1}}(\tau\mathcal{L} _{\kappa}) \left[ \alpha_{r-1}^{n}c_{r-1,k}^n \right].
\end{equation}
Similar with \cref{eq:norm ur-u1}, we have
\begin{equation}\label{eq:tempp1}
    \begin{aligned}
        \left\|u^{n+1}-w_{1}^{n}(\tau)\right\| & \le  \tau \left\|\phi_1(\tau\mathcal{L} _{\kappa})\left[ (\alpha_{r-1}^{n}-1)\mathcal{N}(u^{n}) \right]\right\|  \\
        & \quad + \tau\left\|   \sum_{k=1}^{r-1}k!{\phi_{k+1}}(\tau\mathcal{L} _{\kappa})\left[ \alpha_{r-1}^{n}c_{r-1,k}^n \right]\right\|.
    \end{aligned}
\end{equation}
Similar to \cref{eq:noname_3}, we know that
\begin{equation}\label{eq:tempp2}
    \left\|   \sum_{k=1}^{r-1}k!{\phi_{k+1}}(\tau\mathcal{L} _{\kappa})\left[ \alpha_{r-1}^{n}c_{r-1,k}^n \right]\right\|
    \le \frac{1}{2} \sqrt{r-1} \left\|\left|\boldsymbol{c}^{n}_{r-1}\right|\right\|.
\end{equation}

Using \cref{ineq:aP} for $s=0$, we have
\begin{equation}\label{eq:tempp3}
    \begin{aligned}
        \left\|\phi_1(\tau\mathcal{L} _{\kappa})\left[(\alpha_{r-1}^{n}-1)\mathcal{N}(u^{n})\right]\right\|
        &\le   \left\|(\alpha_{r-1}^{n}-1)\mathcal{N}(u^{n})\right\|   \\
        &\le   \left\|P_{r}^{n}(x,s_{*}(x)) - \mathcal{N}(u^n(x))\right\|   \\
        &=  \left(\left( \sum _{k=1}^{r-1} \left(\frac{s_*}{\tau}\right)^{k} c_{r-1,k}^{n} \right)^{2},1\right)^{\frac{1}{2}}   \\
        &\le 
        % \sqrt{r-1}  \left( \sum _{k=1}^{r-1} \left( c_{r-1,k}^{n} \right)^{2},1\right)^{\frac{1}{2}}   \\
        % &= 
        \sqrt{r-1} \left\|\left|{\boldsymbol{c}}^{n}_{r-1}\right|\right\|.
    \end{aligned}
\end{equation}
Combining \cref{eq:tempp1,eq:tempp2,eq:tempp3}, we have
\begin{equation}\label{eq:norm ur-u1 rescaling}
        \left\|u^{n+1}-w_{1}^{n}(\tau)\right\| \le \frac{3}{2}\tau \sqrt{r-1} \left\|\left|{\boldsymbol{c}}^{n}_{r-1}\right|\right\|.
\end{equation}

Thus, similar to \cref{eq:aim_1}, we can infer that if the following inequality
\begin{equation}\label{eq:aim_1_rescale}
\frac{7}{2}  \tau\sqrt{r-1} \left\|\left|\boldsymbol{c}^{n}_{r-1}\right|\right\| \le \left\| w_{1}^{n}(\tau)-u^{n} \right\|
\end{equation}
holds, then $E\left(u^{n+1}\right) -E\left(u^{n}\right) \le 0$. The equation \cref{eq:aim_1_rescale} is the same as \cref{eq:aim_1} only with some change of coefficients, and the following proof is also similar with the proof of Theorem \ref{theorem:energy decay main theorem}. 

We aim to employ mathematical induction to demonstrate that, for any $r \ge 1$ there exists a specific constant, denoted as $\tau_{\max,r}$, such that the inequality \cref{eq:aim_1_rescale} holds true for all $0<\tau \le \tau_{\max,r}$.

Firstly, it is easy to check $\tau_{\max, 1} = + \infty$.
Then we suppose that for fixed $r\geq 1$, there exists certain constant $\tau_{\max,r}>0$, such that \cref{eq:aim_1_rescale} holds for any $\tau \le  \tau_{\max,r}$ and consider the case of $r+1$ based on the inductive hypothesis, i.e., we want to show that, there exists certain positive constant $\tau_{\max,r+1} \le  \tau_{\max,r}$, such that
\begin{equation} \label{eq:aim_r_rescaling}
    \frac{7}{2} \tau\sqrt{r} \left\|\left|\boldsymbol{c}^{n}_{r}\right|\right\| \le \left\| w_{1}^{n}(\tau)-u^{n}\right\|
\end{equation}
holds for any $\tau \le  \tau_{\max,r+1}$.
Following the proof in Theorem \ref{theorem:energy decay main theorem}, we have
\begin{equation} 
    \frac{7}{2}\tau\sqrt{r}\left\|\left|\boldsymbol{c}^{n}_{r}\right|\right\| \le 
    \frac{7\sqrt{r}\kappa\tau}{\sigma_{\min}\left(V_r\right)}   \left(\sum_{k=1}^{r}\left\|w^{n}_{r}(a_{r,k}\tau)-u^{n}\right\|^{2}\right)^{\frac{1}{2}}. 
\end{equation}
According to the explicit formula of ETDRK$r$ with the rescaling technique, for $k=1,2, \ldots ,r$, we have
\begin{equation}  \label{tempeq1}
\begin{aligned}
    w^{n}_{r}(a_{r,k}\tau)-u^{n} 
    &= w^{n}_{1}(a_{r,k}\tau) - u^{n} 
    + \tau \phi_1(a_{r,k}\tau\mathcal{L} _{\kappa})[(\alpha_{r-1}^{n}-1)\mathcal{N}(u^{n})] 
    \\ & \quad + \tau \sum_{j=1}^{r-1} j! \left(a_{r,k}\right)^{j+1} \phi_{j+1}\left(a_{r,k}\tau\mathcal{L}_{\kappa}\right) \left[\alpha_{r-1}^{n}c^{n}_{r-1,j}\right] .
\end{aligned}
\end{equation}
By \cref{eq:noname_5,eq:tempp3,eq:noname_3}, we have
\begin{equation} \label{329}
\begin{aligned}
    &\frac{7}{2}\tau\sqrt{r} \left\|\left|\boldsymbol{c}^{n}_{r}\right|\right\|  \le 
     \frac{7r\kappa\tau}{\sigma_{\min}\left(V_r\right)}
     \left(\left\|w_{1}^{n}(\tau)-u^{n} \right\|+ \frac{3}{2}\tau \sqrt{r-1}
     \left\|\left|\boldsymbol{c}^{n}_{r-1}\right|\right\| \right).  
\end{aligned}
\end{equation}
Substituting the inductive hypothesis \cref{eq:aim_1_rescale} into \cref{329}, we have
\begin{equation} 
    \begin{aligned} \label{ineq327}
        \frac{7}{2}\tau\sqrt{r} \left\|\left|\boldsymbol{c}^{n}_{r}\right|\right\|  
        \le  \frac{10r\kappa\tau}{\sigma_{\min}\left(V_r\right)}\left\|w_{1}^{n}(\tau)-u^{n} \right\|.    
    \end{aligned}
\end{equation}
From \cref{ineq327}, we know that our aim \cref{eq:aim_r_rescaling} holds for any
\[
    \tau \le  \tau_{\max,r+1} \coloneqq \min \left\{\tau_{\max,r}, \frac{\sigma_{\min}\left(V_r\right)}{10r\kappa}  \right\}.
\]
Since $\tau_{\max,1}=+\infty$,
we have 
\begin{equation} 
    \tau_{\max,r} = \frac{1}{10\kappa} \min \left\{{\sigma_{\min}\left(V_1\right)},\frac{\sigma_{\min}\left(V_2\right)}{2}, \ldots ,\frac{\sigma_{\min}\left(V_{r-1}\right)}{r-1}\right\}.
\end{equation}

Finally, by induction, we finish the proof. 
\end{proof}

\subsection{Convergence analysis}\label{section:Convergence analysis}

The MBP-preserving property plays an important role in analyzing the convergence of ETDRK$r$ schemes. It establishes a priori $L^{\infty}$ bounds on the numerical solutions, which notably simplifies the task of convergence analysis. The following theorem demonstrates that, when equipped with the rescaling technique, the ETDRK methods do not compromise the accuracy inherent to the standard ETDRK methods.

\begin{theorem}
   Suppose that \cref{eq:f(beta)<0,eq:kappa_cond2} hold. For a fixed terminal time $T>0$, suppose that the exact solution $u(x,t)$ to the Allen--Cahn equation \cref{eq:AC} belongs to $C^r(C(\bar{\Omega});[0, T])$ and the initial value $u^0(x)$ satisfies $\left\|u^0\right\|_{\infty} \leq \beta$, 
    and let $\left\{u_{r}^n\right\}_{n \geq 0}$ be generated by the ETDRK$r$ scheme with the rescaling technique and uniform nodes $\left\{s_{k}= {k}\tau/{r}\right\}_{k=0}^{r}$. Then we have
\begin{equation}\label{eq:error r}
    \left\|u_{r}^n-u\left(t_n\right)\right\|_{\infty} \leq K_r \mathrm{e}^{\kappa J_r T} \tau^r
\end{equation}
for any $\tau>0$, where the constants $J_r, K_r>0$ are independent of $\tau$.
\end{theorem}
\begin{proof}
The high-order ETDRK schemes are constructed through an iterative way. Consider the ETDRK$(r+1)$ scheme equipped with the rescaling technique, with $P_r^n(s)$ the interpolation polynomial of $\mathcal{N}(w_{r}^{n}(s))$, as specified in \cref{eq:vandermonde V_r}. Here, $w_{r}^n(s)$ is the solution obtained from the ETDRK$r$ scheme equipped with the rescaling technique, which preserves the MBP. 
Therefore, we have the following error estimation similar with \cref{eq:same order}, i.e.
\begin{equation}\label{ineqt1}
    \left\|{P}_{r}^{n}(s) - \tilde{P}_{r}^{n}(s)\right\|_{\infty} = O \left(\tau^{r+1}\right) , \quad \forall  s \in [0,\tau].
\end{equation}

Let $e_{r}^{n}:=u^{n}_{r}-u\left(t_{n}\right)$ denotes the error of the $r$th-order ETDRK scheme. We first prove the case of ETDRK1 scheme. In fact, from the previous work \cite{DU-JU-LI-QIAO2021ETD}, it is already known that
\begin{equation}
\left\Vert e_{1}^{n+1}\right\Vert _{\infty } \leq   (1+\kappa \tau )\left\Vert e_{1}^{n}\right\Vert _{\infty } + C_1 \tau ^{2},
\end{equation}
where, $C_1>0$ is a constant independent of the time step $\tau$. By extending this through recursion, we establish:
\begin{equation}
\begin{aligned}
    \left\|e_1^n\right\|_{\infty} 
    &\le (1+\kappa \tau)^n\left\|e_1^0\right\|+C_1 \kappa \tau^2 \sum_{k=0}^{n-1}(1+\kappa \tau)^k \\
    &= C_1 \tau\left[(1+\kappa \tau)^n-1\right] \le C_1 \mathrm{e}^{\kappa n \tau} \tau.
\end{aligned}
\end{equation}
By letting $K_1 = C_1$, we obtain \cref{eq:error r} for $r=1$ since $T = n\tau $. 

Suppose that there exist positive constants $M_r, C_{r,1}, C_{r,2}, \cdots, C_{r,r}$ independent of $\tau$, such that
\begin{equation}\label{eq:truncation error}
    \left\Vert e_{r}^{n+1}\right\Vert _{\infty } \le \left( 1+ C_{r,1}\kappa\tau +\cdots + C_{r,r} \left(\kappa\tau\right)^r \right)\left\Vert e_{r}^{n}\right\Vert _{\infty } +M_{r} \tau ^{r+1}.
\end{equation}
Then, based on the inductive hypothesis, we will prove that 
there exist positive constants $M_{r+1}, C_{r+1,1}, C_{r+1,2}, \cdots, C_{r+1,r+1}$ independent of $\tau$, such that
\begin{equation}\label{eq:truncation error r+1}
    \left\Vert e_{r+1}^{n+1}\right\Vert _{\infty } \le  \left( 1+ C_{r+1,1}\kappa\tau +\cdots + C_{r+1,r+1} \left(\kappa\tau\right)^{r+1} \right)\left\Vert e_{r+1}^{n}\right\Vert _{\infty } +M_{r+1} \tau ^{r+2}.
\end{equation}

By Lemma \ref{lem:contraction semi-group}, we have
\begin{align*}
    \left\|e_{r+1}^{n+1}\right\| _{\infty} &=  \left\|\mathrm{e}^{\tau \mathcal{L}_{\kappa}} e_{r+1}^n+\int_0^\tau  \mathrm{e}^{(\tau-s) \mathcal{L}_{\kappa}}\left(\tilde{P}^{n}_r(s)-\mathcal{N}\left[u\left(t_n+s\right)\right]\right) \mathrm{~d} s \right\| _{\infty} 
    \\ 
    & \leq   \mathrm{e}^{-\kappa \tau }\left\Vert e_{r+1}^{n}\right\Vert _{\infty} \\
    & \quad +\int _{0}^{\tau }\mathrm{e}^{-\kappa (\tau -s)}\left(\left\Vert {P}^{n}_r(s) -\prod _{r}\mathcal{N} [u(t_{n} +s)]\right\Vert _{\infty}+R_{r+1}\right)\mathrm{~d} s,
\end{align*}
where $\prod \limits _{r}$ represents the Lagrange interpolation operator corresponding to the uniform nodes $\left\{s_{k} \coloneqq {k}\tau/{r}\right\}_{k=0}^{r}$, and $R_{r+1}$ is the truncation error given by
\begin{equation}
    R_{r+1}(s) = \left\|\mathcal{N} [u(t_{n} +s)]-\prod _{r}\mathcal{N} [u(t_{n} +s)]\right\|_{\infty}+ \left\| \tilde{P}_{r}^{n}(s)-{P}_{r}^{n}(s)\right\|_{\infty}.
\end{equation}
By the error estimates of interpolation, we have
\begin{equation}\label{ineqt2}
        \left\|\mathcal{N} [u(t_{n} +s)]-\prod _{r}\mathcal{N} [u(t_{n} +s)]\right\|_{\infty} = O(\tau^{r+1}).
\end{equation}
By \cref{ineqt1,ineqt2}, we know that
there exists a constant $\tilde{M}_{r}>0$ independent of $\tau$, such that
\begin{equation}\label{eq:R_bound}
     R_{r+1}(s) \le \tilde{M}_{r} \tau^{r+1}, \quad \forall s \in [0,\tau].
\end{equation}
If we denote $\ell_{r,k}(s)$ as the Lagrange basis functions associated with the uniform interpolation nodes $\left\{s_{k}={{k}\tau}/{r}\right\}_{k=0}^{r}$, 
then the interpolation polynomial $P_r^n(s)$ can be expressed as
\[
P_r^n(s) = \sum _{k=0}^{r}\ell_{r,k} (s) \mathcal{N}\left( w_{r}^{n}\left(s_k\right)\right). 
\]
For the Lagrange basis functions $\ell_{r,k}(s)$, it is easy to check that
\begin{equation}\label{eq:lagrange_basis_bound}
    \left|\ell_{r,k}(s)\right|= \left|\prod_{\substack{i=0 \\ i\neq k}}^{r} \frac{s-s_i}{s_k-s_i}\right| \le \left(\frac{\tau}{\tau/r}\right)^{r}=r^{r}.
\end{equation}
By the way, the inequality \cref{eq:lagrange_basis_bound} is not optimal, and one can find smaller upper bound for a fixed $r$.
From \cref{eq:truncation error}, we know (replacing $\tau$ with $s_k$ and $e_r^n$ with $e_{r+1}^n$) that
\begin{equation}\label{eq:truncationr}
    \left\Vert w_{r}^{n}\left(s_k\right) - u(t_{n} +s_{k} )\right\Vert _{\infty } \le \left( 1+  \sum_{i=1}^{r} C_{r,i} \left(\kappa s_k\right)^i \right)\left\Vert e_{r+1}^{n}\right\Vert _{\infty } +M_{r}  s_k ^{r+1}.
\end{equation}
By \cref{eq:R_bound,eq:lagrange_basis_bound}, and using the fact that $1-\mathrm{e}^{-a} \leq a$ for any $a>0$, we have
\begin{align}\label{eq:results}
   \notag \left\|e_{r+1}^{n+1}\right\| _{\infty}  & \le \mathrm{e}^{-\kappa \tau }\left\Vert e_{r+1}^{n}\right\Vert _{\infty} +  \int _{0}^{\tau }\mathrm{e}^{-\kappa (\tau -s)} \tilde{M}_{r} \tau^{r+1} \mathrm{~d} s \\
   \notag &\quad + \int _{0}^{\tau }\mathrm{e}^{-\kappa (\tau -s)} \sum _{k=0}^{r}| \ell_{r,k} (s)| \left\Vert \left(\mathcal{N}\left(w_{r}^{n}\left(s_k\right)\right) -\mathcal{N} \left(u(t_{n} +s_{k} )\right)\right)\right\Vert_{\infty} \mathrm{~d} s \\
    \notag & \le \mathrm{e}^{-\kappa \tau } \left\Vert e_{r+1}^{n}\right\Vert _{\infty}+ +  \int _{0}^{\tau }\mathrm{e}^{-\kappa (\tau -s)} \tilde{M}_{r} \tau^{r+1} \mathrm{~d} s \\
    \notag & \quad + \int _{0}^{\tau }\mathrm{e}^{-\kappa (\tau -s)} r^{r}\sum _{k=0}^{r}2 \kappa \left\Vert w_{r}^{n}\left(s_k\right) -u(t_{n} +s_{k} )\right\Vert_{\infty} \mathrm{~d} s\\
    & \le  \mathrm{e}^{-\kappa \tau }\left\Vert e_{r+1}^{n}\right\Vert _{\infty}+ \frac{1-\mathrm{e}^{-\kappa \tau }}{\kappa} \left( \tilde{M}_{r} \tau^{r+1} + 2 \kappa r^{r} M_{r} \sum _{k=0}^{r}  \left(\frac{k\tau}{r}\right)^{r+1}\right)    \\
   \notag & \quad + \frac{1-\mathrm{e}^{-\kappa \tau }}{\kappa} \left[2 \kappa r^{r}\sum _{k=0}^{r}\left( 1+ \sum_{i=1}^{r} C_{r,i} \left(\kappa\frac{ k\tau}{r}\right)^i \right)\left\Vert e_{r+1}^{n}\right\Vert_{\infty } \right] \\
    \notag& \le \mathrm{e}^{-\kappa \tau } \left\Vert e_{r+1}^{n}\right\Vert _{\infty}+  \tau \left( 2 \kappa  r^{r+1} M_{r} \tau^{r+1} \int_{1/r}^{1+1/r} x^{r+1}\mathrm{~d} x+ \tilde{M}_{r} \tau^{r+1} \right)  \\
    \notag& \quad + 
    \left[ 2 r^{r}(r+1)\kappa\tau  + 2 r^{r}\sum_{i=1}^{r} C_{r,i} \left(\kappa\tau\right)^{i+1} \int_{1/r}^{1+1/r} x^{i}\mathrm{~d} x   \right]\left\Vert e_{r+1}^{n}\right\Vert_{\infty } \\
   \notag & = \left(1+C_{r+1,1}\kappa\tau +\cdots + C_{r+1,r+1} \left(\kappa\tau\right)^{r+1} \right)\left\Vert e_{r+1}^{n}\right\Vert_{\infty } + M_{r+1} \tau^{r+2},
\end{align}
The constants are defined as follows
\begin{align*}
M_{r+1} &\coloneqq \frac{2\kappa r^{r+1} M_r}{r+2} \left[\left(1+\frac{1}{r}\right)^{r+2}-\left(\frac{1}{r}\right)^{r+2}\right] + \tilde{M}_r,\\
C_{r+1,1} &\coloneqq 2r^r(r+1),\\
C_{r+1,k} &\coloneqq \frac{2r^r C_{r,k-1}}{k}  \left[\left(1+\frac{1}{r}\right)^{k}-\left(\frac{1}{r}\right)^{k}\right], \quad k = 2,3,\cdots,r.
\end{align*}
By  mathematical induction, we proved that \cref{eq:truncation error} holds for any $r \ge 1$. 

Finally, by recursion for \cref{eq:truncation error}, we obtain
\begin{align*}
    \left\|e_r^n\right\|  & \leq\left[1+ \sum_{i=1}^{r} C_{r,i} \left(\kappa\tau\right)^i\right]^n\left\|e_r^0\right\|
    + M_r \tau^{r+1} \sum_{k=0}^{n-1}\left[ 1+ \sum_{i=1}^{r} C_{r,i} \left(\kappa\tau\right)^i \right]^k \\
    & =  M_r \tau^{r+1} \frac{ \left[ 1+ \sum_{i=1}^{r} C_{r,i} \left(\kappa\tau\right)^i \right]^n  -1}{\sum_{i=1}^{r} C_{r,i} \left(\kappa\tau\right)^i}  
     \le  K_r \tau^{r}  \mathrm{e}^{n J_r \kappa\tau} ,
\end{align*}
where $K_r \coloneqq M_r/\left(\kappa C_{r,1}\right)$,
$\displaystyle J_r \coloneqq \max_{1\le k \le r} k! C_{r,k} %\left\{ \right\}
$ and $T = n\tau$.
\end{proof}

\section{Numerical experiments}\label{section:numerical experiments}
Consider the Allen--Cahn equation \cref{eq:AC} in the domain $ \Omega = (0,2\pi) \times (0,2\pi)$ with homogeneous Neumann boundary conditions, and use uniform interpolation nodes to generate ETDRK schemes.
We implement a uniform rectangular mesh with a mesh size of $h$ to partition the domain, and we use the central finite difference method to discrete the Laplace operator. 
% In this case, the product of the discrete Laplace matrix exponential with a vector can be efficiently implemented by using the Fast Fourier Transform (FFT). 
Since the discrete Laplace operator $\Delta_h$ generated by the central differencing method also forms a generator of a contraction semigroup, the proof of the fully discrete MBP still holds by replacing the $\Delta$ with $ \Delta_{h}$ in Lemma \ref{lem:contraction semi-group}. 
% \textcolor{red}{
% The contraction semigroup plays an important role in preserving the MBP for ETDRK schemes. However, matrices derived from high-order finite difference methods typically cannot form a generator of a contraction semigroup. Consequently, when employing finite difference methods for spatial discretization, our approach is limited to second-order in space to ensure the MBP. For high-order spatial discretizations, the finite element method has proven to be highly effective \cite{kovacs2017numerical}.
% }

We choose two commonly used forms of potential function $F$. The first is the Ginzburg--Landau potential function, defined as
\begin{equation}\label{eq:F_GL poly}
    F_{\text{GL}}(u)=\frac{1}{4}\left(1-u^2\right)^2.
\end{equation}
The corresponding $f$ is $f_{\text{GL}}(u) = u-u^3$.
The maximum bound is $\beta=1$, and
the stabilizing constant is $\kappa\ge \max _{|\xi| \le \beta}\left|f_{GL}^{\prime}(\xi)\right| = 2$.
The second is the Flory-Huggins potential, given by
\begin{equation}\label{eq:F_FH log}
    F_{\text{FH}}(u)=\frac{\theta}{2}[(1+u) \ln (1+u)+(1-u) \ln (1-u)]-\frac{\theta_c}{2} u^2,
\end{equation}
where $\theta$ and $\theta_c$ are constants satisfying $0<\theta<\theta_c$, and the corresponding $f$ is
\begin{equation}\label{eq:f_FH log}
    f_{\text{FH}}(u) =\frac{\theta}{2} \ln \frac{1-u}{1+u}+\theta_c u.
\end{equation}
 $\beta$ is the positive root of $f_{\text{FH}}(u) = 0$.
In the following numerical experiments, we set $\theta=0.8$ and $\theta_c=1.6$. Then $\beta \approx 0.9575$, and $\kappa\ge \max _{|\xi| \le \beta}\left|f_{FH}^{\prime}(\xi)\right| \approx  8.02$.

%On this mesh, we generate random data within the range of $-\beta$ to $\beta$ to establish the initial configuration of $u$. 

\subsection{Convergence in time}
Consider the Allen--Cahn equation \cref{eq:AC} with $\varepsilon = 0.1$ and the cubic function $f_{\text{GL}}=u-u^{3}$ from the Ginzburg--Landau potential.
To verify the temporal convergence rates of the ETDRK$r$ with the rescaling technique schemes, let us consider the smooth initial value $u^{0}(x,y) = 0.5 \sin x \sin y$.
We set a uniform spatial mesh size $h=2 \pi / 512$ and the terminal time $T = 2$.
With these settings, we calculate the numerical solutions
with various time step sizes $\tau =0.1 \times  2^{-k}, k =0, 1, \ldots , 6$ and calculate the relative errors to get the convergence rate.
The $L^{\infty}$ and $L^{2}$ norms are considered to calculate the convergence rates.
It can be observed in Table \ref{table:combined1} that the convergence rate approaches theoretical values.

\begin{table}[htbp]
  \centering
  \footnotesize
  \caption{Rates of convergence for third-order to fifth-order ETDRK schemes with the rescaling technique.}\label{table:combined1}  
  \begin{center}
  \begin{tabular}{|ccccc|}
    \hline
    \multicolumn{5}{|c|}{ETDRK3 with the rescaling technique}                                                                            \\ \hline
    \multicolumn{1}{|c|}{$\tau = 0.1$} & \multicolumn{1}{c|}{$L^{\infty}$ error} & \multicolumn{1}{c|}{Rate} & \multicolumn{1}{c|}{$L^{2}$ error} &  Rate \\ \hline
    \multicolumn{1}{|c|}{$\tau$} & \multicolumn{1}{c|}{5.343e+00} & \multicolumn{1}{c|}{-} & \multicolumn{1}{c|}{1.843e-02} & - \\ 
    \multicolumn{1}{|c|}{$\tau /2$} & \multicolumn{1}{c|}{9.554e-01} & \multicolumn{1}{c|}{2.484} & \multicolumn{1}{c|}{3.316e-03} & 2.475 \\ 
    \multicolumn{1}{|c|}{$\tau /4$} & \multicolumn{1}{c|}{1.439e-01} & \multicolumn{1}{c|}{2.731} & \multicolumn{1}{c|}{4.999e-04} & 2.729 \\ 
    \multicolumn{1}{|c|}{$\tau /8$} & \multicolumn{1}{c|}{1.978e-02} & \multicolumn{1}{c|}{2.863} & \multicolumn{1}{c|}{6.873e-05} & 2.863 \\ 
    \multicolumn{1}{|c|}{$\tau /16$} & \multicolumn{1}{c|}{2.593e-03} & \multicolumn{1}{c|}{2.931} & \multicolumn{1}{c|}{9.012e-06} & 2.931 \\ 
    \multicolumn{1}{|c|}{$\tau /32$} & \multicolumn{1}{c|}{3.320e-04} & \multicolumn{1}{c|}{2.965} & \multicolumn{1}{c|}{1.154e-06} & 2.965 \\ \hline 
    \hline
    \multicolumn{5}{|c|}{ETDRK4 with the rescaling technique}                                                                            \\ \hline
    \multicolumn{1}{|c|}{$\tau = 0.1$} & \multicolumn{1}{c|}{$L^{\infty}$ error} & \multicolumn{1}{c|}{Rate} & \multicolumn{1}{c|}{$L^{2}$ error} &  Rate \\ \hline
    \multicolumn{1}{|c|}{$\tau$} & \multicolumn{1}{c|}{1.161e+00} & \multicolumn{1}{c|}{-} & \multicolumn{1}{c|}{4.035e-03} & - \\ 
    \multicolumn{1}{|c|}{$\tau /2$} & \multicolumn{1}{c|}{1.070e-01} & \multicolumn{1}{c|}{3.440} & \multicolumn{1}{c|}{3.722e-04} & 3.438 \\ 
    \multicolumn{1}{|c|}{$\tau /4$} & \multicolumn{1}{c|}{8.154e-03} & \multicolumn{1}{c|}{3.714} & \multicolumn{1}{c|}{2.837e-05} & 3.714 \\ 
    \multicolumn{1}{|c|}{$\tau /8$} & \multicolumn{1}{c|}{5.631e-04} & \multicolumn{1}{c|}{3.856} & \multicolumn{1}{c|}{1.959e-06} & 3.856 \\ 
    \multicolumn{1}{|c|}{$\tau /16$} & \multicolumn{1}{c|}{3.701e-05} & \multicolumn{1}{c|}{3.928} & \multicolumn{1}{c|}{1.288e-07} & 3.928 \\ 
    \multicolumn{1}{|c|}{$\tau /32$} & \multicolumn{1}{c|}{2.372e-06} & \multicolumn{1}{c|}{3.964} & \multicolumn{1}{c|}{8.252e-09} & 3.964 \\ \hline 
    \hline
    \multicolumn{5}{|c|}{ETDRK5 with the rescaling technique}                                                                            \\ \hline
    \multicolumn{1}{|c|}{$\tau = 0.1$} & \multicolumn{1}{c|}{$L^{\infty}$ error} & \multicolumn{1}{c|}{Rate} & \multicolumn{1}{c|}{$L^{2}$ error} &  Rate \\ \hline
    \multicolumn{1}{|c|}{$\tau$} & \multicolumn{1}{c|}{2.011e-01} & \multicolumn{1}{c|}{-} & \multicolumn{1}{c|}{7.002e-04} &  -\\ 
    \multicolumn{1}{|c|}{$\tau /2$} & \multicolumn{1}{c|}{9.390e-03} & \multicolumn{1}{c|}{4.421} & \multicolumn{1}{c|}{3.270e-05} & 4.420 \\ 
    \multicolumn{1}{|c|}{$\tau /4$} & \multicolumn{1}{c|}{3.593e-04} & \multicolumn{1}{c|}{4.708} & \multicolumn{1}{c|}{1.251e-06} & 4.708 \\ 
    \multicolumn{1}{|c|}{$\tau /8$} & \multicolumn{1}{c|}{1.242e-05} & \multicolumn{1}{c|}{4.854} & \multicolumn{1}{c|}{4.327e-08} & 4.854 \\ 
    \multicolumn{1}{|c|}{$\tau /16$} & \multicolumn{1}{c|}{4.085e-07} & \multicolumn{1}{c|}{4.927} & \multicolumn{1}{c|}{1.423e-09} & 4.927 \\ 
    \multicolumn{1}{|c|}{$\tau /32$} & \multicolumn{1}{c|}{1.310e-08} & \multicolumn{1}{c|}{4.963} & \multicolumn{1}{c|}{4.574e-11} & 4.959 \\ \hline 
    \end{tabular}
  \end{center}
\end{table}

% \begin{table}[htbp]
%   \centering
%   \footnotesize
%   \caption{Rates of convergence for ETDRK6 with rescaling technique.}\label{table:c6}
%   \begin{tabular}{|c|c|c|c|c|}
%   \hline
%   $\tau = 0.1$ & $L^{\infty}$ error & Rate & $L^{2}$ error & Rate \\ \hline
%   $\tau $ & 7.239e-03 & - & 1.009e-04 & - \\  
%   $\tau /2$ & 1.700e-04 & 5.412e+00 & 2.369e-06 & 5.412e+00 \\ 
%   $\tau /4$ & 3.255e-06 & 5.707e+00 & 4.537e-08 & 5.707e+00 \\  
%   $\tau /8$ & 5.628e-08 & 5.854e+00 & 7.845e-10 & 5.854e+00 \\  
%   $\tau /16$ & 9.256e-10 & 5.926e+00 & 1.294e-11 & 5.922e+00 \\  
%   $\tau /32$ & 1.682e-11 & 5.782e+00 & 3.265e-13 & 5.309e+00 \\ \hline 
%   \end{tabular}
% \end{table}

\subsection{Unconditional preservation of the MBP}
Consider the Allen--Cahn equation \cref{eq:AC} with $\varepsilon = 0.1$ and the logarithmic function $f_{\text{FH}}$ from the Flory--Huggins potential. The preservation of MBP is important in this case since the equation consists of the logarithmic terms which will involve complex numbers if the value of the solution is out of the interval $(-1, 1)$. 
We set a uniform time step $\tau=1$,  a uniform spatial mesh size $h=2 \pi / 512$ and a random data ranging from  $-\beta$ to $\beta$ generated on the mesh as the initial value $u^{0}$, which is highly oscillated.
For $r=3,5,7$, we compare the maximum norm of solutions generated by the standard ETDRK$r$ schemes and the ETDRK$r$ schemes with the rescaling technique in Figure \ref{figure:MBP}.
From Figure \ref{figure:MBP}, it can be seen that the maximum norm of numerical solutions of standard ETDRK3, ETDRK5, and ETDRK7 exceed the maximum bound $\beta$, and after using the rescaling technique, the numerical solutions preserve the MBP, with errors of the same magnitude as standard ones.

\begin{figure}[!ht]
  \centerline{
  \hspace{-0.3cm}
  \includegraphics[width=0.34\textwidth]{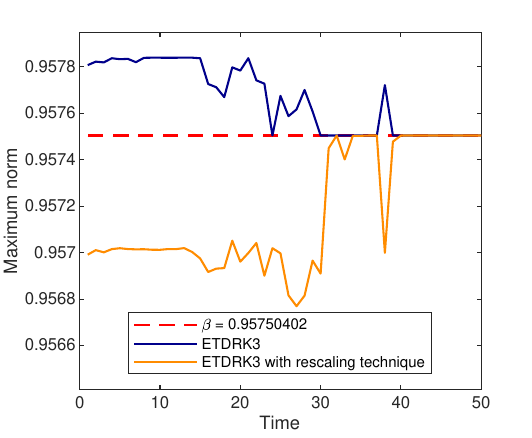}\hspace{-0.2cm}
  \includegraphics[width=0.34\textwidth]{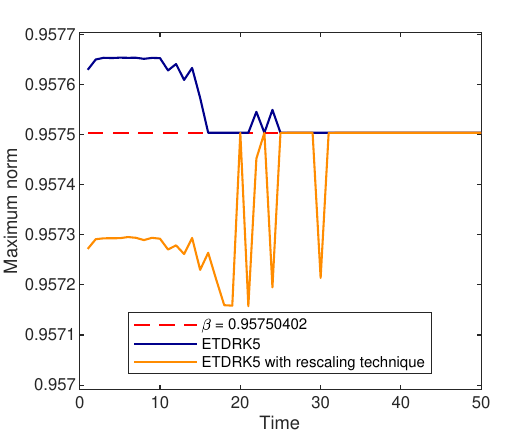}\hspace{-0.2cm}
  \includegraphics[width=0.34\textwidth]{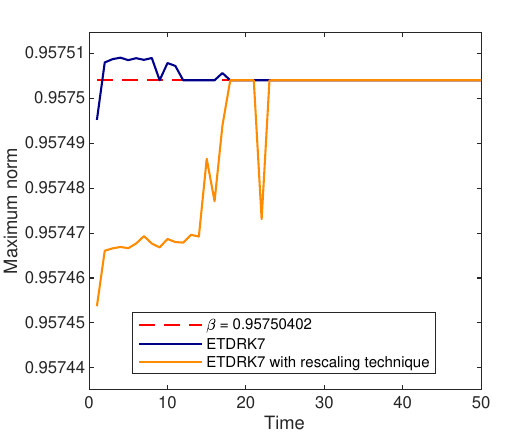}}
  \vspace{-0.2cm}
  \caption{Evolutions of the maximum norms of the numerical solutions for $r=3,5$ and $7$, respectively (left to right).}
  \label{figure:MBP}
\end{figure}

\subsection{Original energy dissipation law}
Consider the Allen--Cahn equation \cref{eq:AC} with $\varepsilon = 0.1$ and the logarithmic function $f_{\text{FH}}$. 
We set a uniform spatial mesh size $h=2 \pi / 512$ and the smooth initial value $u^{0}(x,y) = 0.5 \sin x \sin y$.
For $r=3,4,5,6$,
we compute the original energy of numerical solutions generated by the ETDRK$r$ schemes equipped with the rescaling technique in Figure \ref{figure:MBP} with $\tau = 0.2$, $0.1$, and $0.01$.
\begin{figure}[!ht]
  \centerline{
  \hspace{-0.3cm}
  \includegraphics[width=0.34\textwidth]{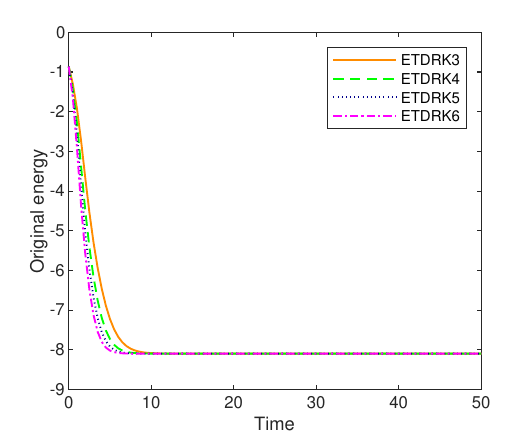}\hspace{-0.2cm}
  \includegraphics[width=0.34\textwidth]{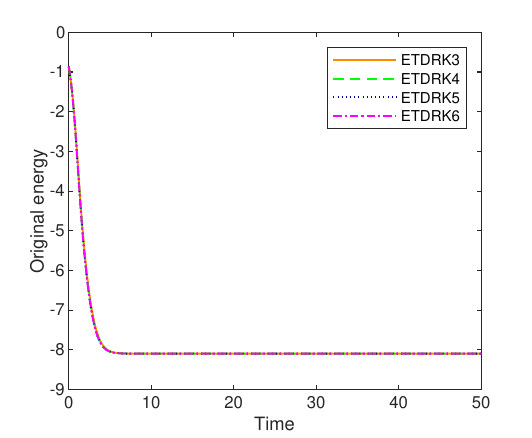}\hspace{-0.2cm}
  \includegraphics[width=0.34\textwidth]{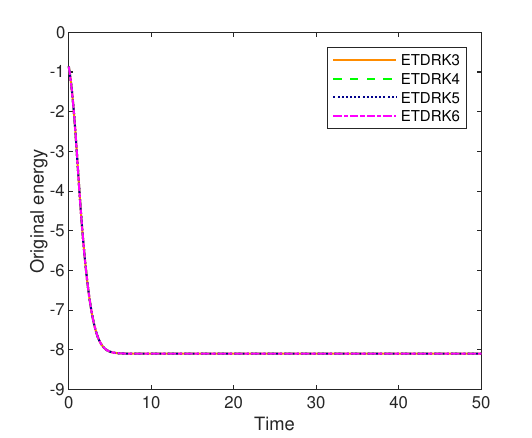}}
  \vspace{-0.2cm}
  \caption{Comparison of original energy with different schemes but same time steps for $\tau=0.2$, $0.1$ and $0.01$, respectively (left to right).}
  \label{fig:energy_same_timestep}
\end{figure}
According to Theorem \ref{theorem:energy decay rescaling}, the time-step size constraints are $\tau_{\max,3} = 1.031 \times 10^{-3}$, $\tau_{\max,4} = 1.141  \times10^{-4}$, $\tau_{\max,5} = 1.374\times10^{-5}$ and $\tau_{\max,6} = 1.697\times10^{-6}$.
However, from Figure \ref{fig:energy_same_timestep}, we can see that the ETDRK schemes do not require a very strict time-step size to maintain the original energy dissipation law. For ETDRK6, our theoretical time-step size restriction is about $10^{-6}$, but Figure \ref{fig:energy_same_timestep} shows that the original energy still decreases with $\tau=0.2$ for this example.

\section{Conclusions}\label{section:conclusions}

    We analyze the MBP and original energy dissipation law of arbitrarily high-order ETDRK schemes for Allen--Cahn equations. 
    We propose some time-step size restrictions to preserve the original energy dissipation law when the nonlinear term is Lipschitz continuous. 
    In addition, we have proposed a rescaling technique to preserve the MBP unconditionally without influencing accuracy of the numerical solution, which can guarantee the Lipschitz condition on the nonlinear term.
    Moreover, our analysis is suitable for arbitrarily high-order ETDRK schemes.
    We also provide some numerical examples to verify theoretical results and show that the ETDRK schemes with rescaling technique have better properties than the standard one.
    In future studies, we will expect to preserve the original energy dissipation law without any time step restriction and extend our energy analysis and the rescaling technique to more phase-field models.

    % In this paper, we analyze the MBP and original energy dissipation law of arbitrarily high-order ETDRK schemes for Allen--Cahn equations. 
    % We propose some sufficient conditions to preserve the original energy dissipation law and find the absolute time step size restrictions for decreasing the original energy. 
    % As for the MBP, we have proposed a rescaling technique to preserve the MBP restriction on the time-step size, which keeps the same order convergence as the standard scheme. We have proved that the ETDRK$r$ scheme with rescaling technique still preserves the original energy dissipation law within some restriction of time-step size and has $r$th-order convergence in time. 
    % In addition, all of our analysis are applicable for arbitrarily high-order ETDRK schemes. 
    % We also provide some numerical examples to verify the theoretical results.
    % In future studies, we will expect to preserve the original energy dissipation law without any time step restriction and extend our energy analysis and the rescaling technique to more phase-field models.

% \appendix
% \section{An example appendix} 
% \lipsum[71]

% \begin{lemma}
% Test Lemma.
% \end{lemma}

%\section*{Acknowledgments}

\bibliographystyle{siamplain}
\bibliography{references}
\end{document}